\documentclass[11pt,reqno]{amsart}

\usepackage[utf8]{inputenc}
\usepackage{amsfonts,latexsym,amssymb,amsmath,amsthm,amsrefs,etex,slashed,cancel}
\usepackage{todonotes}
\usepackage[hidelinks]{hyperref}
\usepackage{a4wide}
\usepackage{enumerate}
\usepackage{MnSymbol}
\usepackage{nicefrac}

\newtheorem{thm}{Theorem}
\newtheorem*{theorem*}{Theorem}
\newtheorem{prop}{Proposition}
\newtheorem{lem}{Lemma}
\newtheorem{cor}{Corollary}
\newtheorem{remark}{Remark}

\newtheorem{defn}{Definition}

\newtheorem{example}[thm]{Example}

\newcommand{\Courant}[0]{\operatorname{Co}}

\numberwithin{equation}{section}

\begin{document}
\title[Approximation of water hammer by a Lax-Wedroff finite difference scheme]{
Approximation of the non-linear water hammer problem by a Lax-Wendroff finite difference scheme 
%Lax-Wendroff scheme for the water hammer problem
}
\author[Hugo Carrillo-Lincopi]{Hugo Carrillo-Lincopi}
\address{Inria Chile Research Center, Av. Apoquindo 2827, Las Condes, Santiago, Chile} \email{hugo.carrillo@inria.cl}

\author[A. Waters]{Alden Waters}
\address{Leibniz Universit\"{a}t Hannover, Institut f\"{u}r Analysis,
Welfengarten 1,
30167 Hannover,
Germany
} \email{alden.waters@math.uni-hannover.de}

\author[T. Xu ]{Teke Xu}
\address{ University of Groningen, Bernoulli Institute,
Nijenborgh 9,
9747 AG Groningen,
The Netherlands
} \email{t.xu@rug.nl}

\maketitle

\begin{abstract}
We study the water hammer problem in the case of a sudden closing of a valve upstream, and we consider a Lax-Wendroff finite difference scheme in order to obtain a numerical solution of this problem. In order to establish the approximation of this scheme to the original case, we rigorously show some properties such as consistency, stability and weak convergence of the scheme under reasonable conditions. In addition, we present some numerical simulations in order to show some features of the numerical method.
\end{abstract}
%\textcolor{blue}{draft version do not distribute}

\section{Introduction}
Water flow can be described by various models such as the Saint-Venant model \cite{BastinCoron2009}, a variant of the linearized Euler equations \cite{ChrisShuang2014,Gugat2018}, and a switched differential algebraic equation model \cite{KausTren18}. Although the former two models follow the same general equation set, the Saint-Venant model usually refers to shallow water systems with free surfaces, while a variant of the linearized Euler equations sometimes called the isothermal Euler equations \cite{Gugat2018} typically refer to fluid flow in a (closed) pipeline system. In \cite{BastinCoron2016book} the nomenclature `Saint-Venant equations' is used interchangeably for the fluid flow in a lake and a pipe. They refer to the difference as ``open channel" and ``pipeline model". Here the main motivation for our study of the Saint-Venant equations is the so-called water hammer problem which occurs when a valve is closed very quickly in a pipeline carrying water. Let $c_f$ and $D$ be constants which reflect the friction and diameter of the pipeline respectively. We refer to our model as the Saint-Venant equations (pipeline case) in the spirit of \cite{BastinCoron2016book}, Example 1.6.1, and these are given by: 
\begin{align}\label{pdep1}
\begin{cases}
\partial _t\rho + \partial_x q = 0 \\
\partial_tq + \partial_x\left(\frac{q^2}{\rho} + p\right) = -\frac{c_fa|q|q}{2D}, \quad (t,x)\in (0,T)\times \mathbb{R},\\
(\rho,q)=(\rho_0,q_0)
\end{cases}\end{align}
where $p$ is pressure which is directly proportional to $\rho$ the material density and $q$ is the fluid flow with $(\rho_0,q_0)$ some given functions. The number $a=0$ if the model considers no friction \cite{BastinCoron2016book} and $a=1$ if it considers friction \cite{Herty2009, Gugat2018}. 

Several numerical methods have been used to approximate these equations \eqref{pdep1} with $a=0$, such as the Upwind Scheme cf. \cite{AudussePerthatme2004, WuZhao2015} and Galerkin methods cf. \cite{Williams2019, LiuTadmor2008}. For higher-order accuracy numerical method, one can also refer to essentially non-oscillatory (ENO) schemes \cite{Harten1989, HartenOsher1986, Harten1997} or weighted-essentially non-oscillatory (WENO) schemes \cite{LiuOsher1994, Shu1998}. However it has been proposed in the presence of friction in the pipeline, that there should be an extra term (corresponding to $a=1$) as stated in \cite{Herty2009, zhangboran2018}. To this end, we study the nonlinear hyperbolic balanced laws with changing boundary conditions and semi-linear friction term. The theorems and numerical algorithms here are the first to rigorously analyze pipe valve closure as a dynamical boundary condition. 

For simplicity, if we assume a pipe that is $2L$ long with left and right side connecting to a reservoir (giving constant water density) and at time $t_1$, there is a valve which closes instantly to cut off the water flow at $x = L$ in the pipe, and this leads to changing boundary conditions. The water hammer problem occurs when the valve is closed very quickly. We assume in our model that the valve is closed instantaneously leading to a hard boundary condition. This scenario is different to the slow closing valves c.f. \cite{BastinCoron2016book} for an overview or the engineering model \cite{KausTren18} which gets rid of the spatial variable. Therefore, we introduce an indicator function (similar to a Heaviside function) $\mathbb{H}{(\Omega)}$ which maps from some subset $\Omega\subset [0,T]\times [0,2L]$ to the set $\{0,1\}$, to indicate that whether the flow is constrained. In this case we take $\mathbb{H}$ to be $1$ on $\Omega=[t_1,T]\times [0,L]\cup [0,t_1)\times [0,2L]$ and $0$ elsewhere. This indicates that we have restricted the equation set to hold on the left hand side of the pipe after the valve closes. We now modify \eqref{pdep1} to capture these dynamics so we are analyzing the equation set:
\begin{subequations}\label{pdep2}
\begin{align}
&(\partial _t\rho + \partial_x q)\cdot\mathbb{H}(\Omega) = 0,\\
&\left[\partial_tq + \partial_x\left(\frac{q^2}{\rho} + p\right)\right]\cdot\mathbb{H}(\Omega) = \left(-\frac{c_fq|q|}{2D\rho}\right)\cdot\mathbb{H}(\Omega),\\
&(\rho(0,x),q(0,x))=(\rho_0(x),q_0(x)),\\
&\rho(t,0) = \rho_1(t) \quad t\in[0,T], \label{bcleft}\\
&\rho(t,2L) = \rho_2(t) \quad t\in[0,t_1),\quad q(t,L)=0 \quad t\in [t_1,T]. \label{bcright}
\end{align}
\end{subequations}
Functions $(\rho_0,q_0)$ and $(\rho_1,\rho_2)$ are given. The boundary condition \eqref{bcright} indicates the hard boundary of the valve - that is there is no new flow at $x=L$ after a certain time. The boundary condition \eqref{bcleft} indicates the reservoir. We assume throughout this paper that the system is hyperbolic before the valve closes, and hyperbolic in the weak sense after the valve closes. The definition of hyperbolic is reviewed below equation \eqref{hivp} in the next section. 

In the appendix of \cite{BastinCoron2016book}, Theorem B.1 shows the well-posedness of the systems through Riemann coordinates, such that if the conditions $(\rho_0,q_0)$ and $(\rho_1,\rho_2)$ are given the system \eqref{pdep2} for some $t_1$ the system is well posed with $(\rho,\frac{q}{\rho})$ in $C^{\infty}([0,t_1)\times [0,2L])^2$. Therefore in our analysis we assume $(q_0,\rho_0)$ are chosen so that $u=(\rho,\frac{q}{\rho})$ with $\rho\neq 0$ is in $C^{\infty}([0,t_1)\times [0,2L])^2$ before the valve closure and in $H^{-1}([t_1,T]\times [0,L])^2$ after the valve closure, which given \cite{BastinCoron2016book} is a good space to make sense of the solution. This assumption is discussed more throughly in Section \ref{regularitydiscussion}. Through abuse of notation we write $C^{\infty}([0,t_1)\times [0,2L])^2\cup H^{-1}([t_1,T]\times [0,L])^2$   as the space of the solution even though we mean the spaces before and after the valve closure. We construct a numerical Lax-Wendroff scheme for \eqref{pdep2} which has a rigorous error analysis over the engineering ODE model proposed by \cite{SteinebachRosenSohr2012, KausTren18} with modelling based on method of characteristics. This new dynamical equation set \eqref{pdep2} could work for gas or petroleum pipelines industry c.f. \cite{dafermos2005hyperbolic, stiriba2002nonlinear, baudin2005relaxation}, and could also refer to blood flowing through artificial valves in the human heart \cite{gomez2017analysis,quateroni2006,liangfuyou2009,Luskin1982THEEO}. 

The main contribution of this paper is therefore twofold. The first is consistency and stability of the Lax-Wendroff scheme with the inclusion of a term which represents friction. It should be noted in \cite{Gallouet2022} that general principles for Lax-Wendroff type convergence of a finite volume scheme have been established for Euler equations in low dimensions, and in the case of Saint-Venant in 2d \cite{Herbin2021}. Our method is different in this specific case and also because of the presence of the friction term. This  term (\eqref{pdep1}, $a=1$) makes it analytically more difficult to prove convergence.  

The second and most important contribution is the analysis of the valve closure as being an instantaneous boundary condition, formulated by \eqref{pdep2}. While shocks have been analyzed extensively, fast valve closure has not appeared in the literature before. It was unclear to the authors how using another finite volume scheme such as in \cite{Gallouet2022} or \cite{Herbin2021} that the rapid valve closure could be incorporated. We show that the numerical solution still converges in the weak sense. This result is exciting because it shows that rapidly moving valves, such as pipe and heart valves, can be analyzed this way in the future. %
% These numerical techniques will be applicable to more general (conservative) systems of the form:
%\begin{align}\label{hyperbolic}
%\frac{\partial u}{\partial t}=\sum\limits_{j=1}^dA_j(u,x)\frac{\partial u}{\partial x_j}+B(x,u), 
%\end{align}
%where $u(x,t)=u$ is a p-component vector with initial value $u(x,0)=u_0(x)$ coupled with changing boundary conditions. A major furture goal is to expand the results to multiple valves opening and closing-modeling the human circulatory system or blood flow in the body.\\
%
%In general, we propose a finite difference scheme with non-linear Lax-Wendroff method, and show the convergence of the scheme by proving the consistency and the stability (with the notion of $l_2$-stability) of the numerical solution, and when there is a sudden change of the boundary condition, the numerical solution of the non-linear Lax-Wendroff scheme still converges to a analytical solution (if it exists) in the sense of distributions.  

The outline of this article is as follows: In Section \ref{schemeS}, we rewrite the system in a quasi-linear form and describe the finite difference operator we use. We introduce the main theorems in Section \ref{mainS}. In Section \ref{CS}, we show the consistency and the stability of the scheme which was proposed in Section 2, and prove Theorem \ref{main}. We specify in Section \ref{bcS} the numerical version of the boundary conditions that are needed before and after the influence of the valve on the system. In Section \ref{weakS}, we prove the convergence of the numerical solution to the analytical solution with discontinuity in distributional sense, and this is our second main result which is Theorem \ref{main2}. Finally, in Section \ref{simul} we present some simulations by applying the conditions discussed in Section \ref{bcS} and show the convergence we achieved. 

\section{Numerical scheme setup}\label{schemeS}
In this section, we first consider the general Saint-Venant model with semilinear friction term without boundary conditions, which means $-\infty<x<\infty$, and it also works when a pipe connects to 2 reservoirs at each end, discussed in Section \ref{bcS}. The effect of the valve switching closed (instantaneously changing boundary condition) will be discussed later also in \ref{bcS}.\\

Consider the equation \eqref{pdep1} with $a=1$, If we substitute $q$ with $\rho \cdot v$ into (\ref{pdep2}), and with the linear pressure law $p(\rho) = p_a + K\frac{\rho - \rho_a}{\rho_a}$, we have the partial differential equation
\begin{align}\label{pde3}
\partial_t\left(\begin{matrix}
\rho\\v
\end{matrix}\right) + \left(\begin{matrix}
v&\rho\\\frac{K}{\rho_a\rho}&v
\end{matrix}\right)\partial_x\left(\begin{matrix}
\rho\\v
\end{matrix}\right) = \left(\begin{matrix}
0\\-\frac{c_f}{2D}v|v|
\end{matrix}\right).
\end{align}
Let $u = \left(\begin{matrix}
\rho\\v
\end{matrix}\right)$, $A(u) = \left(\begin{matrix}
v&\rho\\\frac{K}{\rho_a\rho}&v
\end{matrix}\right)$, and $B(u) = \left(\begin{matrix}
0\\-\frac{c_f}{2D}v|v|
\end{matrix}\right)$, then we can reformulate our system as
\begin{align}\label{hivp}
\begin{cases}
\partial_tu(t,x) &= -A(u(t,x))\partial_xu(t,x) + B(u(t,x)),\\
u(0,x) &= u_0(x).
\end{cases}
\end{align}
This formulation of the system is needed to apply the theorem in \cite{Strang1964}. We assume that our system is hyperbolic, that is \eqref{hivp} has matrix $A$ with distinct eigenvalues of opposite sign for all $(t,x)$. Let $t = n\Delta t, \ \ x = j\Delta x, \ \ n,j\in \mathbb{N}$ where $U_j^n=U(n\Delta t, j\Delta x)$ stands for the numerical solution for each space and time step.  We want to approximate the equation (\ref{hivp}) by a finite difference scheme $\varphi$:
\begin{align}\label{fds}
U(t+\Delta t, x) = \varphi\left(\{U(t, x \pm  j\Delta x)\}_{j\in\mathbb{N}},\Delta t, t, \Delta x, x\right).
\end{align}
%To make sure the finite difference scheme (\ref{fds}) is accurate we need to prove that
%$$\lim_{\Delta x, \Delta t \rightarrow 0}U(t + \Delta t,x) = u(t,x)$$ if $u(t,x)$ is the actual solution of the initial value problem (\ref{hivp}).
The basic idea of the Lax-Wendroff method is to expand the solution of the equation up to the second order in time $t$, and then to use space derivatives to replace the time derivative. This is done through an intermediate step of constructing a function $U(t,x)$ close to the solution $u(t,x)$, which when evaluated at the grid points is referred to as the numerical solution $U(n\Delta t,j\Delta x)$. Through out this article, we use explicit central schemes for approximating the space derivative, that is for a solution $v$ to a partial differential equation we have 
\begin{equation}\label{centraldiff1}
\partial_xv(t,x) = \frac{v(t,x + \Delta x) - v(t, x - \Delta x)}{2\Delta x} + O((\Delta x)^2),
\end{equation}
\begin{equation}\label{centraldiff2}
	\partial_x^2v(t,x) = \frac{v(t,x + \Delta x) - 2v(t,x) + v(t, x - \Delta x)}{\Delta x^2} + O((\Delta x)^2).
\end{equation}
which requires $v$ to be $C^{\infty}$ if interpreted in the classical sense for repeated use of the finite difference method. However for the nonlinear Lax-Wendroff method the computations are complicated. The complications occur because for the full partial differential equation ($a=1$) it is difficult to compute the convergence conditions. Secondly unless we use \cite{Strang1964} it is not necessarily true as in the linear case that consistency and stability implies convergence.  We recall two definitions that will help us apply Strang's theorem from \cite{Strang1964} in Section \ref{CS}. As part of Section \ref{CS}, we will show the scheme \eqref{fds} is consistent for the equation \eqref{hivp}. First we start by defining an expansion for a candidate $U(t,x)$ approximation of a solution to a smooth quasi-linear hyperbolic equation as follows:
\begin{equation}\label{principalerr}
	U(t,x) := u(t,x) + \sum_{k \ge 1}^\mathcal{K} (\Delta t)^k  V_k(t,x), \ \ \ k = 1,2,3,...,
\end{equation}
where $V_k$ is a series of error terms existing up to some power $\mathcal{K}$, classified by the orders of $(\Delta t)^k$. The power $\mathcal{K}$ depends on what numerical scheme is used. Of course this expansion cannot be true for all values of $t,x$ and $\Delta t$. The cone in which this expansion holds uniformly is dictated by the so-called Courant-Friedrich-Lewy (CFL) condition. 

Therefore from the scheme (\ref{fds}), we know that the discrete time evolution of the solution is derived by a finite difference scheme $\varphi$, with the scheme $\varphi$ being also expressed as
\begin{align}\label{fds2}
	\varphi(\{U(t,x \pm j\Delta x)\}_{j\in\mathbb{N}},\Delta x,\Delta t) &= U(t,x) + \Delta t \partial_tU(t,x) + O((\Delta t)^2)\\ \nonumber
	&= U(t,x) + \Delta t(-A(U)\partial_xU + B(U))+O((\Delta t)^2)
\end{align}
so that we can obtain the truncation error of the scheme by substituting every $U$ in the scheme \eqref{fds2} by (\ref{principalerr}) where
\begin{align*}
U(t+\Delta t,x) = \varphi(\{U(t,x \pm j\Delta x)\}_{j\in\mathbb{N}},\Delta x,\Delta t). 
\end{align*}
with $U(x,0,\Delta t)=u_0(x).$ In order to make this scheme work, instead of the natural numbers, let us a consider a subset $\mathbb{L}$. Let us denote $\mathbb{L}$ the set of points $l_1, . . . ,l_s$ where $\mathbb{L}$ is a subset of the set $\mathcal{L}$ of lattice points with integer coordinates in space. Let $\mathbb{L}^h$ be the convex hull of $\mathbb{L}$ in this space and fix the ratio $r=\Delta t/\Delta x$. Then given a point $(t_0,x_0)$ we define the closed cone 
\begin{align}
\mathcal{C}_0=\{(x,t)|\, 0\leq t\leq t_0,\,\, r(x-x_0)/(t_0-t)\in \mathbb{L}^h\}
\end{align}
and the intersection with the mesh of width $\Delta x$. Then we get 
\begin{align}
\mathcal{M}(t,\Delta x)=\{(t,x)| \,(t,x)\in \mathcal{C}_0,\,\, x=x_0+l\Delta x, l\in \mathcal{L}\}
\end{align}
Calculating $U(t_0,x_0,\Delta t)$ where usually the $\Delta t$ is suppressed where it is understood, involves only the points in $\mathcal{M}(t_0-i\Delta t,\Delta x)$ with $i=1, . . ,t_0/\Delta t$ and that as $\Delta t\rightarrow 0$ these points are dense in $\mathcal{C}_0$. Under some hypothesis on $\mathcal{C}_0$ we want to show 
\begin{align}
\lim\limits_{\Delta t\rightarrow 0}U(t_0,x_0,\Delta t)=u(t_0,x_0).  
\end{align}
%which is
%\begin{align*}
%u(t+\Delta t,x) + \sum_{k \ge 1}^K&(\Delta t)^kV_k(t + \Delta t,x) = u(t,x) + \sum_{k \ge 1}^K(\Delta t)^kV_k(t,x)\\ \nonumber
%& + \Delta t\bigg[-A(u(t,x))\partial_x\bigg (u + \sum_{k \ge 1}^K(\Delta t)^kV_k(t,x)\bigg) \\ \nonumber
%& + B\bigg(u + \sum_{k \ge 1}^K(\Delta t)^kV_k(t,x)\bigg)\bigg] + O((\Delta t)^2).
%\end{align*}
To this end we start by Taylor expanding $B(\cdot)$
\begin{align}\label{nonlinexpan}
B\left(u + \sum_{k \ge 1}^\mathcal{K} (\Delta t)^kV_k(t,x)\right) = B(u) + \frac{\partial B(u)}{\partial u}\cdot\left(\sum_{k \ge 1}^\mathcal{K} (\Delta t)^kV_k\right) + \beta_k(t,x)
\end{align}
and then substituting (\ref{nonlinexpan}) into (\ref{hivp}), and using \eqref{principalerr} and the definition of $\varphi$ we have a series of systems classified by the order of $(\Delta t)^k$ such that
\begin{equation}
\begin{cases}
	V_k(t+\Delta t,x) = V_k(t,x) - \Delta tA(u)\partial_xV_k(t,x) + \Delta t\frac{\partial B(u)}{\partial u}V_k + \Delta t\cdot \beta_k(t,x)\\
	V_k(0,x) = 0, \ \ \  k = 1,2,3,...
\end{cases}
\end{equation}
Here $\beta_k$ is the remainder of the Taylor's expansions from (\ref{nonlinexpan}), being a nonlinear inhomogeneous term for each system of $V_k$. We recall the first of our two definitions:
\begin{defn}
The order of accuracy $p$ of the difference operator is the index $j$ of the first non-vanishing principal error term $V_j(t,x)$.
\end{defn} 
To make sure the finite difference scheme is $\ell_2$-stable, we must verify that the $\ell_2$-norm of its first variation is bounded. Therefore we introduce the Gateaux derivative of a vector function. The finite difference scheme we use here is a function of a series of $U(t,x + j\Delta x)$, in addition to the space and time step $\Delta x$ and $\Delta t$. So we let
\begin{equation}\label{Gateaux}
	d\varphi(\mathcal U) := \lim_{\epsilon \rightarrow 0}\frac{\varphi(\mathcal U + \epsilon f) - \varphi(\mathcal U)}{\epsilon},
\end{equation}
where $\varphi \in \mathbb{R}^2$, 
\begin{align*}
\mathcal U(t,x):= (\rho(t,x - \Delta x), \rho(t,x), \rho(t,x + \Delta x), v(t,x - \Delta x),v(t,x), v(t,x + \Delta x))^T,
\end{align*}
\begin{align*}
f := (f_1(x - \Delta x), f_1(x), f_1(x + \Delta x),f_2(x - \Delta x),f_2(x),f_2(x + \Delta x))^T.
\end{align*}
Now we recall the following definition from \cite{Strang1964}: 
\begin{defn}\label{defnfv}
	Let $c_j$ be the Jacobian of $\varphi$ in \eqref{fds} with respect to its $j^{th}$ argument $U(t,x + j\Delta x)$, $j = 1,2,3,...$, then the first variation of $\varphi$ with respect to the argument vector $U(t,x)$ is defined by 
	\begin{equation}\label{Jac}
	(M(t,\Delta x) f)(x) := d\varphi(U) = \sum_{j=1}^{s}c_j(u(t,x),...,u(t,x),x,t,0)f(x + j\Delta x).
	\end{equation}
\end{defn}
We can think of $f$ and $Mf$ as defined on $\mathcal{M}(t,\Delta x)$ and $\mathcal{M}(t+\Delta t,\Delta x)$ respectively. 

\section{Statement of the Main Theorems}\label{mainS} 

The first theorem is a generic convergence result of the scheme on the problem \eqref{pdep1} or \eqref{pdep2} before the valve closure. Let $L^*>0$ be a real number. 
\begin{thm}\label{main}
Assume the solution $u=(\rho,\frac{q}{\rho})$ to \eqref{pdep1} with $a=1$  is in $ C^{\infty}([0,T]\times [0, L^*])^2$. Let $\varphi$ be the finite difference operator of the 1d quasi-linear hyperbolic equation \eqref{hivp} corresponding to the approximate numerical solution $U(t,x)$ of the real solution $u(t,x)$ which satisfies the iteration 
	\begin{equation*}
	U(t+\Delta t, x) = \varphi\left(\{U(t, x \pm j \Delta x)\}_{j\in\mathbb{N}},\Delta t, t, \Delta x, x\right).
	\end{equation*}
The finite difference scheme $\varphi$ is given by the Lax-Wendroff method with $u(t + \Delta t,x)$ which is Taylor expandable in $t$ up to the second order
	\begin{equation*}
		u(t + \Delta t, x) = u(t,x) + \Delta t\cdot \partial_tu(t,x) + \frac{\Delta t^2}{2}\partial_t^2u(t,x) + O((\Delta t)^3)
	\end{equation*}
on $[0,T]\times [0, L^*]$. Then the finite difference operator $\varphi$ is consistent with accuracy order $p = 2$, and its first variation is $\ell_2$-stable with proper CFL-condition on $\Delta t$ and $\Delta x$. Furthermore the constructed $U(t,x)$ is equal to the actual solution modulo an error, that is
	\begin{equation}
		U(t_0,x_0) = u(t_0,x_0) + O((\Delta t)^2).
	\end{equation}
for $(t_0,x_0)$ inside the cone dictated by the CFL condition.  
\end{thm}
This is the first time the Lax-Wendroff scheme has been shown to be consistent and accurate for the Saint-Venant or pipeline equation with a semi-linear friction term. The reason for this choice of numerical model/technique is that there is no straightforward way of showing consistency and stability implies convergence for this problem without using Strang's methods \cite{Strang1964}. 

For the convenience of computation, we put the equation set \eqref{hivp} in divergence form by setting
\begin{align*}
	w := \left(\begin{matrix}
	\rho\\q
	\end{matrix}\right),\ 
	F(w) := \left(\begin{matrix}
	q\\q^2/\rho + p(\rho)
	\end{matrix}\right),\ 
	G(w) := \left(\begin{matrix}
	0\\\frac{-c_fq|q|}{2D\rho}
	\end{matrix}\right).
\end{align*}
We refer to $u=(\rho,\frac{q}{\rho})$ from Theorem \ref{main} as the normalized solution. 
Since Theorem \ref{main} describes the dynamics before the valve closure, we now need to describe the dynamics after the valve closure. 
\begin{thm}\label{main2}
Let $w\in H^{-1}([t_1,T]\times [0,L])^2\cup C^{\infty}([0,t_1)\times [0,2L])^2$, be a  solution of the PDE \eqref{pdep2} that satisfies 
\begin{align*}
&\int\limits_{t_1}^T\int\limits_{0}^L w\cdot\partial_t\varphi + F(w)\cdot\partial_x\varphi- G(w)\cdot \varphi \,dx\,dt + \int\limits_0^{L}w(t_1,x)\cdot\varphi(t_1,x)dx \\&+ \nonumber
	 \int\limits_0^Lw(t_1,x)\cdot\varphi(t_1,x)dx + \int\limits_{t_1}^TF(w(t,L))\varphi(t,L)dt = 0
\end{align*}
	for any test function $\varphi(t,x) = (\varphi_1(t,x)\ \varphi_2(t,x))^\top$ with $\varphi_1,\varphi_2 \in C_0^\infty([0,T]\times [0,2L])$. Let $\{U_j^n\}_{(n,j)\in I}$, $I=\{0,1,2, . . . ,N_x\}\times \{0,1,2, . . . ,N_t\}$ be a set of numerical solutions computed from the Lax-Wendroff scheme (corresponding to $u$) with the initial condition vector given by,
	$$\{U_j^0\}_{0 \le j \le N_x} = \{\lim\limits_{t\rightarrow t_1^-}u(t,j\Delta x)\}_{0 \le j \le N_x},$$ 
At the time of closing of the valve the scheme is updated using the conditions
	\begin{align*}
	\{[U_{0}^n]_1\}_{0 \le n \le N_t} = \{u_1(n\Delta t,0)\}_{0 \le n \le N_t}, \quad \{[U_{N_x}^n]_2\}_{0 \le n \le N_t} = 0,\quad \lim\limits_{t\to t_1^-}\partial_x\rho(t,L)=0
	\end{align*} 
where $u_1$ is given by the reservoir at $x=0$ and $\lim\limits_{t\to t_1^-}\partial_x\rho(t,L)=0$ evaluated at the grid points.
 Suppose there exists a constant $C_h$ such that 
	\begin{align*}
	\lim\limits_{t\to t_1^-}\|u(t,\cdot)\|_{L^\infty([0,L])} \le C_h, \quad 
	|U^n_j| \le C_h\ \ \  \forall n,j \in \mathbb{N}, \quad 
		\lim\limits_{t \rightarrow t_1^-}\|\partial_xF(w)\|_{L^1([0,2L])} \le C_h,
	\end{align*}
	then the numerical solution set $\{U_j^n\}_{(n,j)\in I}$ converges to the weak solution $u(t,x)$ of order 1 in the distributional sense when $\Delta t, \Delta x \rightarrow 0$ under the generalized CFL condition, $\frac{\Delta t}{\Delta x} \le C_{CFL}$ for some constant $C_{CFL} < \infty$, specified in Proposition \ref{mainprop} and computed explicitly.  
\end{thm}
The specification of the numerical conditions to match the reservoir is given in Subsection \ref{BCbefore}. 

\section{Dynamics before the valve closure}\label{CS}
In this section, we will analyze  the consistency of the finite difference scheme $\varphi$ in (\ref{fds}) and the stability of the approximation solution by Lax-Wendroff method, which will result in the proof of Theorem \ref{main}.  Historically, it has already been proved that under a proper CFL condition, the Lax-Wendroff scheme is consistent and stable when applied to some linear partial differential equations, such as the advection equation \cite{Allairebook2007}, but in the nonlinear area it has not been fully used yet, especially with changing boundary conditions. Also unlike linear PDEs, \cite{Tadmor2012review} pointed out that the challenges may vary between different nonlinearities to prove the equivalence of stability and convergence, but still linear stability carries over to the nonlinear setup according to the celebrated theorem of Strang \cite{Strang1964}. Therefore, we only need the implication that consistent and $\ell_2$-stable properties imply convergent in order to conclude convergence of our modified Lax-Wendroff scheme. This implication is an explicit application of an abstract result of  \cite{Strang1964}. This section follows the notation of \cite{Strang1964} closely. 

By Taylor expanding, we have that
\begin{equation}\label{titeration}
	\begin{cases}
	\rho(t + \Delta t,x) = \rho(t,x) + \Delta t \cdot \partial_t\rho(t,x) + \frac{(\Delta t)^2}{2}\partial^2_t\rho(t,x) + O((\Delta t)^3)\\
	v(t + \Delta t,x) = v(t,x) + \Delta t \cdot \partial_tv(t,x) + \frac{(\Delta t)^2}{2}\partial^2_tv(t,x) + O((\Delta t)^3).
	\end{cases}
\end{equation}
Again we have that
\begin{equation}\label{firsttd}
	\begin{cases}
	\partial_t\rho = -\rho\partial_xv - v\partial_x\rho\\
	\partial_tv = -v\partial_xv - \frac{K}{\rho_a}\frac{1}{\rho}\partial_x\rho - Cv|v|,
	\end{cases}
\end{equation}
so the second derivatives of $t$ derived from (\ref{firsttd}) are
\begin{equation}\label{secondtd}
\begin{cases}
	\partial^2_t\rho = 2\rho(\partial_xv)^2 + 2\rho v\partial^2_xv + 4v\partial_xv\partial_x\rho + 2C|v|\rho\partial_xv + Cv|v|\partial_x\rho \\
	+ \left(\frac{K}{\rho_a} + v^2\right)\partial^2_x\rho\\
	\partial^2_tv = 2v(\partial_xv)^2 + \left(\frac{K}{\rho_a} + v^2\right)\partial^2_xv + 5Cv|v|\partial_xv + \frac{2K}{\rho_a\rho}\partial_x\rho\partial_xv - \frac{2Kv}{\rho_a\rho^2}(\partial_x\rho)^2 \\
	+ \frac{2Kv}{\rho_a\rho}\partial_x^2\rho + 2C|v|\frac{K}{\rho_a\rho}\partial_x\rho + 2Cv^3.
\end{cases}
\end{equation}
Therefore by substituting (\ref{firsttd}) and (\ref{secondtd}) back into (\ref{titeration}) and using the central difference scheme (\ref{centraldiff1}) and (\ref{centraldiff2}), we have the finite difference scheme (\ref{fds}), i.e., we can construct $\varphi\left(\{U(t, x \pm j \Delta x)\}_{j\in \mathbb{N}},\Delta t,  \Delta x\right)$. Now we have the following proposition: 
\begin{prop}\label{main2prop}
	Assume the solution to \eqref{pde3} is in $u\in C^{\infty}([0,T]\times [0,L])^2$. %where $O\subset \mathbb{R}$ is an open interval.
 The nonlinear Lax-Wendroff scheme on \eqref{hivp} is consistent with accuracy of order two on $[0,T]\times [0,L]$.
\end{prop} 
\begin{proof}
	First, it is easy to verify that the Lax-Wendroff scheme for $\rho$ and $v$ is consistent from the computations \eqref{titeration},\eqref{firsttd} and \eqref{secondtd}. Then we can consider the expansion for the specific $U(t,x)=(P(t,x),V(t,x))$ as described previously by \eqref{principalerr} so that 
	\begin{align}\label{expanSV}
			P(t,x,\Delta t) &= \rho(t,x) + \sum^K_{k\ge 1}(\Delta t)^k \varrho_k,\\
			V(t,x,\Delta t) &= v(t,x) + \sum^K_{k\ge 1}(\Delta t)^k \upsilon_k. \nonumber
	\end{align}
	Then we can insert the expansion (\ref{expanSV}) into the finite difference operator $$\varphi\left(\{U(t, x \pm j \Delta x)\}_{j\in\mathbb{N}},\Delta t, t, \Delta x, x\right)$$ to get a Lax-Wendroff scheme on the pair of solutions $\rho(t,x)$ and $v(t,x)$ where
	\begin{align*}
		\rho(t + \Delta t, x) + \sum^K_{k\ge 1}(\Delta t)^k\varrho_k(t + \Delta t, x) &= \rho(t,x) + \sum^K_{k\ge 1}(\Delta t)^k\varrho_k(t,x) \\
		&+ \xi(\Delta t, \Delta x, \rho(t,x), v(t,x), \varrho_k, \upsilon_k) \\ \nonumber
	v(t + \Delta t, x) + \sum^K_{k\ge 1}(\Delta t)^k\upsilon_k(t + \Delta t, x) &= v(t,x) + \sum^K_{k\ge 1}(\Delta t)^k\upsilon_k(t,x) \\
	&+ \eta(\Delta t, \Delta x, \rho(t,x), v(t,x), \varrho_k, \upsilon_k)\nonumber ,
	\end{align*}
    and $\xi(\cdot)$, $\eta(\cdot)$ denote the error terms associated to these expansions.
	%For simplicity, we are using the functions $\xi(\cdot)$ and $\eta(\cdot)$ to replace a string of terms from the finite difference operator $\varphi$ which are complicated from the terms corresponding to $a=1$ in \eqref{pdep1}. We can with patience compute each term in the scheme and classify them with orders of $\Delta t$. 
    However, we should bear in mind that the Lax-Wendroff scheme solves the PDE only up to $(\Delta t)^2$, so that the nonlinear term in the series of the systems we should care about only occurs when $k = 1$ and $k = 2$. We examine these leading order terms as follows:	
	
	$Case \ \ k = 1$ we have 
	\begin{align*}
	\partial_t\rho_1(t,x) &= A_1(\rho,v,\partial_xv)\partial_x\rho_1(t,x) + \alpha_1(\cdot)\rho_1(t,x) + \beta_{1,1},\\
	\partial_tv_1(t,x) &= A_2(\rho,v,\partial_xv)\partial_xv_1(t,x) + \alpha_2(\cdot)v_1(t,x) + \beta_{2,1},\\
	\rho_1(0,x) & = 0,\\
	v_1(0,x) & = 0,
	\end{align*}
	and also	\\
	$Case \ \ k = 2$ we have
	\begin{align*}
	\partial_t\rho_2(t,x) &= A_1(\rho,v,\partial_xv)\partial_x\rho_2(t,x) + \alpha_1(\cdot)\rho_2(t,x) + \beta_{1,2},\\
	\partial_tv_2(t,x) &= A_2(\rho,v,\partial_xv)\partial_xv_2(t,x) + \alpha_2(\cdot)v_2(t,x) + \beta_{2,2},\\
	\rho_2(0,x) & = 0,\\
	v_2(0,x) & = 0.
	\end{align*}
	From the computation of each cases we can find that $\beta_{1,1} = \beta_{2,1} = 0$ but $\beta_{1,2} \neq 0$, $\beta_{2,2} \neq 0$, which indicates in the systems that $\rho_1 = v_1 = 0$ but $\rho_2$ and $v_2$ are not vanishing. So we conclude the order of the accuracy is $p = 2$.
\end{proof}
\subsection{Proof of Theorem \ref{main}}
We start by computing the Gateaux derivatives so that we can eventually compute the Jacobian. For the convenience of writing computations, we write $f^+$ and $f^-$ short for $f(x + \Delta x)$ and $f(x - \Delta x)$. Also because $\varphi$ is a 2-element column vector, we denote each of them as $\varphi(\mathcal{U})_{[1]}$ and $\varphi(\mathcal{U})_{[2]}$. By definition of the Gateaux derivative (\ref{Gateaux}) we have that 
\begin{align*}
d\varphi(\mathcal{U})_{[1]} = &f_1 + \frac{\Delta t}{2\Delta x}\big(-\rho(f_2^+ - f_2^-) - f_1(v^+ - v^-) - v(f^+_1 - f_1^-) - f_2(\rho^+ - \rho^-)\big)\\
& + \frac{(\Delta t)^2}{(\Delta x)^2}\left(\frac{1}{2}\rho(v^+ - v^-)(f_2^+ - f_2^-) + \frac{1}{4}f_1(v^+ - v^-)^2\right)\\
& + \frac{(\Delta t)^2}{(\Delta x)^2}\big(\rho v(f^+_2 - 2f_2 + f_2^-) + (f_1v + f_2\rho)(v^+ - 2v + v^-)\big)\\
& + \frac{(\Delta t)^2}{(\Delta x)^2}\left(\frac{1}{2}v(v^+ - v^-)(f^+ - f_1^-) + v(f_2^+ - f_2^-)(\rho^+ - \rho^-)+f_2(v^+ - v^-)(\rho^+ - \rho^-)\right)\\
& + \frac{C(\Delta t)^2}{2\Delta x}\big(\pm f_2\rho(v^+ - v^-) \pm vf_1(v^+ - v^-) \pm \rho v (f^+_2 - f_2^+)\big)\\
& + \frac{C(\Delta t)^2}{4\Delta x}\big(-v^2(f_1^+ - f_1^-) - 2f_2(\rho^+ - \rho^-)\big)\\
& + \frac{(\Delta t)^2}{(\Delta x)^2}\left(\frac{1}{2}\left(\frac{K}{\rho_a} + v^2\right)(f^+_1 - 2f_1 + f_1^-) + 2vf_2(\rho^+ - 2\rho + \rho^-)\right).
\end{align*}
\begin{align*}
d\varphi(\mathcal{U})_{[2]} = &f_2 + \frac{\Delta t}{2\Delta x}(-v(f_2^+ - f_2^-) - f_2(v^+ - v^-) + \frac{K}{\rho_a\rho}(f_1^+ - f_1^-) - \frac{K}{\rho_a\rho^2}f_1(\rho^+ - \rho^-)) \\
& + \frac{(\Delta t)^2}{4(\Delta x)^2}(f_2(v^+ - v^-)^2 - 2v(v^+ - v^-)(f_2^+ - f_2^-))\\
& + \frac{(\Delta t)^2}{2(\Delta x)^2}\left(\left(\frac{K}{\rho_a} + v^2\right)(f_2^+ - 2f_2 + f_2^-) + 2vf_2(v^+ - 2v + v^-)\right)\\
& + \frac{5C(\Delta t)^2}{4\Delta x}|v|\big(2f_2(v^+ - v^-) + v(f_2^+ - f_2^-)\big)\\
& + \frac{K(\Delta t)^2}{4\rho_a(\Delta x)^2}\left(\frac{1}{\rho}(\rho^+ - \rho^-)(f_2^+ - f_2^-) - \frac{f_1}{\rho^2}(v^+ - v^-)(f_1^+ - f_1^-)\right)\\
& + \frac{(\Delta t)^2}{8(\Delta x)^2}\left(\frac{-4Kv}{\rho_a\rho^2}(\rho^+ - \rho^-) + \frac{4Kvf_1}{\rho_a\rho^3}(\rho^+ - \rho^-)^2 - \frac{2Kf_2}{\rho_a\rho^2}(\rho^+ - \rho^-)^2\right)\\
& + \frac{K(\Delta t)^2}{\rho_a(\Delta x)^2}\left(-\frac{vf_1}{\rho^2}(\rho^+ - 2\rho + \rho^-) + \frac{v}{\rho}(f_1^+ - 2f_1 + f_1^-) + \frac{f_2}{\rho}(\rho^+ - 2\rho + \rho^-)\right)\\
& + \frac{CK(\Delta t)^2}{2\rho_a\Delta x}\left(\frac{|v|}{\rho}(f_1^+ - f_1^-) \pm \frac{f_2}{\rho}(\rho^+ - \rho^-)\right)\\
& + 2C\Delta t|v|f_2 + 3C(\Delta t)^2v^2f_2.
\end{align*}
Then we set $v(t,x + \Delta x) = v(t,x) = v(t,x - \Delta x)$ and $\rho(t,x + \Delta x) = \rho(t,x) = \rho(t,x - \Delta x)$, which corresponds to those stated in the Definition \ref{defnfv} and we have
\begin{align*}
d\varphi(\mathcal{U})_{[1]} &= f_1 + \frac{\Delta t}{\Delta x}(-\rho(f_2^+ - f_2^-) - v(f_1^+ - f_1^-))\\ \nonumber
& + \frac{C(\Delta t)^2}{2\Delta x}(\rho|v|(f_2^+ - f_2^-) - 1/2v^2(f_1^+ - f_1^-))\\ \nonumber
& + \frac{(\Delta t)^2}{(\Delta x)^2}\left(1/2(\frac{K}{\rho_a} + v^2)(f_1^+ - 2f_1 + f^-_1)\right),\\
d\varphi(\mathcal{U})_{[2]} &= f_2 + \frac{\Delta t}{2\Delta x}\left(-v(f_2^+ - f_2^-) + \frac{K}{\rho_a\rho}(f_1^+ - f_1^-) \right)\\
	& + \frac{(\Delta t)^2}{\Delta x}\left(\frac{5C}{4}v|v|(f_2^+ - f_2^-) + \frac{CK|v|}{2\rho_a\rho}(f_1^+ - f_1^-)\right)\\
	& + \frac{(\Delta t)^2}{(\Delta x)^2}\left(\frac{1}{2}\left(\frac{K}{\rho_a} + v^2\right)(f_2^+ - 2f_2 + f_2^-) + \frac{Kv}{\rho_a\rho}(f_1^+ - 2f_2 + f_1^-)\right)\\
	& + 2C\Delta t|v|f_2 + 3C(\Delta t)^2v^2f_2.
\end{align*}
Recall the definition of the Jacobian from \eqref{Jac}. Usually, the Jacobian corresponds to the discrete version of the operator operator $A$, but now we also have the terms from $B$ in \eqref{hivp}. Now that we have carefully computed the first variations for the finite difference operator $\varphi$, we recall the following lemma 
\begin{lem}[Strang, \cite{Strang1964}]\label{lemmal2}
	If for some constant $C$ which is independent of $\Delta t, \Delta x$ but may depend on $t$, we have
	\begin{align*}
		\|M(t - \Delta t,\Delta x)M(t - 2\Delta t,\Delta x)...M(t - n\Delta t,\Delta x)\|_{\ell_2} \le C, \ \ n\Delta t \le t \le t_0,
	\end{align*}
then the first variation of the Lax-Wendroff scheme \eqref{fds} is $\ell_2$-stable.
\end{lem}
%\begin{lem}[Strang, \cite{Strang1964}]\label{lemmabound}
%	Suppose that $\|\prod\limits^n_m M_i\| \le C_M$, $C_M \ge 1$, for $1 \le m \le n \le N$, and that $\|S_i\| \le \xi$, $1 \le i \le N$. Then we have
%	\begin{align}
%	\|\prod\limits_n^N(M_i + S_i)\| \le C_M(1 + \xi C_M)^N \le C_M\exp(\xi C_MN).
%	\end{align}
%\end{lem}
This Lemma %corresponds
agrees
with the modern definition of $\ell_2$ stability c.f. Remark 2.2.19 in \cite{Allairebook2007} but allows for the extra term $B$ in equation \eqref{hivp}. In Strang's proof in \cite{Strang1964}, the term $B$ does not contain any derivatives in $u$ so it does not change the stability of the operator. 
In order to use this Lemma, we introduce some concepts from \cite{Allairebook2007}. We need to define a norm for the numerical solution $U^n=(U_j^n)_{1\leq j\leq N}$, where $U_j^n=U(n\Delta t, j\Delta x)$. We take the classical $\ell_p$ norm which we scale by the space step $\Delta x$ to obtain 
\begin{align}
\|U^n\|_p=\left(\sum\limits_{j=1}^N\Delta x|U_j^n|^p\right)^{\frac{1}{p}}\quad 1\leq p\leq \infty
\end{align}
Thanks to the weighting by $\Delta x$ the norm $\|U^n\|_p$ is identical to the norm $L^p(0,1)$ for piecewise constant functions over subintervals $[x_j,x_{j+1}]$ of $[0,1]$. We therefore refer to this as the $L^p$ norm. We want to show for the nonlinear problem that the finite difference Jacobian can be bounded as an operator on $\ell_2$. Using the norm introduced above, we want to prove the following proposition:
\begin{prop}\label{mainprop}
Let $\varepsilon<1$ be a small number depending on the $L^{\infty}$ norm of the solution, $u$ to \eqref{hivp}. The first variation of the Lax-Wendroff scheme $\varphi$ as stated in \eqref{fds} of the quasi-linear hyperbolic equation \eqref{hivp} is conditionally $\ell_2$-stable with a CFL condition:
	\begin{equation}
		\frac{\Delta t}{\Delta x}\sqrt{\frac{K}{\rho_a}} \le \varepsilon(\|u\|_{L^{\infty}})<1
	\end{equation}
In particular since $\varepsilon$ is sufficiently small then this implies
	\begin{equation}
		\|M\|_{L^2}\le 1 + O_{L^\infty}\left(\frac{\Delta t}{\Delta x}\sqrt{\frac{K}{\rho_a}}\right)^2<2.
	\end{equation}
The subscript $L^\infty$ refers to the dependence of the constant in the $O$ terms on the $L^\infty$ norm of the solution to \eqref{hivp} in the classical sense.
\end{prop}
\begin{proof}	
Let $\Courant$ be the Courant number such that $\Courant:= \frac{\Delta t}{\Delta x}\sqrt{\frac{K}{\rho_a}}$, and $\sqrt{\frac{K}{\rho_a}}$ is the sound speed. Let $C:= \frac{c_f}{2D}$ and let $M(f)$ be the first variation of the Lax-Wendroff scheme with respect to $U(t,x + j\Delta x)$, j = 1,2,3..., then we have
\begin{align*}
&|M(f)_{[1]}| \le |f_1| + \frac{\Delta t}{\Delta x}\left[\|\rho\|_{L^\infty}(|f^+_2| + |f_2^-|) + \|v\|_{L^\infty}(|f_1^+| + |f_1^-|)\right] \\&+ \frac{C(\Delta t)^2}{\Delta x}\left[\|\rho\|_{L^\infty}\|v\|_{L^\infty}(|f_2^+|+|f_2^-|) + \frac{1}{2}\|v\|_{L^\infty}(|f_1^+| + |f_1^-|)\right]\\
& + \frac{(\Delta t)^2}{(\Delta x)^2}\left[1/2(\frac{K}{\rho_a} + \|v\|^2_{L^\infty})(|f_1^+| + 2|f_1| + |f_1^-|)\right].
\end{align*}
Similarly we have
\begin{align*}
&|M(f)_{[2]}| \le |f_2| + \frac{\Delta t}{\Delta x}\left[\|v\|_{L^\infty}(|f^+_2| + |f_2^-|) + \frac{K}{\rho_a} \|1/\rho\|_{L^\infty}(|f_1^+| + |f_1^-|)\right]\\
& + \frac{(\Delta t)^2}{\Delta x}[\frac{5C}{4}\|v\|^2_{L^\infty}(|f_2^+|+|f_2^-|) + \frac{K}{\rho_a}\|v/\rho\|_{L^\infty}(|f_1^+| + |f_1^-|)]\\
& + \frac{(\Delta t)^2}{(\Delta x)^2}[\frac{1}{2}(\frac{K}{\rho_a} + \|v\|^2_{L^\infty})(|f_2^+| + 2|f_2| + |f_2^-|)]\\
&  + \frac{(\Delta t)^2}{(\Delta x)^2}[\frac{K}{\rho_a}\|v/\rho\|_{L^\infty}(|f_1^+| + 2|f_1| + |f_1^-|)].
\end{align*}
Let $r:=\frac{\Delta t}{\Delta x}$ and we denote the subscript $L^{\infty}$ in the $O_{L^\infty}$ terms to depend only on the $L$-infinity norm of the solution to \eqref{hivp}. We take $L^2$ norm of both sides of $|M(f)_{[1]}|$ and $|M(f)_{[2]}|$ to get
\begin{align*}
\|M(f)_{[1]}\|_{L^2} &\le \|f\|_{L^2}\bigg[1 + \frac{2\Delta t}{\Delta x}\big(\|\rho\|_{L^\infty} + \|v\|_{L^\infty}\big) \\ 
& + \frac{C(\Delta t)^2}{\Delta x}\left(\|\rho\|_{L^\infty} + \|v\|_{L^\infty} + \frac{1}{2}\|v\|^2_{L^\infty}\right) + \frac{2(\Delta t)^2}{(\Delta x)^2}\left(\frac{K}{\rho_a} + \|v\|_{L^\infty}\right)\bigg]\\
& = 1 + O_{L^\infty}(r + r^2) + 2\Courant^2 + O_{L^\infty}(\Delta t),
\end{align*}
\begin{align*}
\|M(f)_{[2]}\|_{L^2} &\le \|f\|_{L^2}[1 + \frac{\Delta t}{\Delta x}(\|v\|_{L^\infty}  + \frac{K}{\rho_a}\|1/\rho\|_{L^\infty})\\
&+ \frac{C(\Delta t)^2}{\Delta x}(\frac{5}{2}\|v\|^2_{L^\infty} + \frac{K}{\rho_a}\|\frac{v}{\rho}\|_{L^\infty})+ \frac{(\Delta t)^2}{(\Delta x)^2}(\frac{2K}{\rho_a} + 2\|v\|^2_{L^\infty} + \frac{4K}{\rho_a}\|\frac{v}{\rho}\|_{L^\infty})]\\
& = 1 + O_{L^\infty}(r + r^2) + \Courant^2(2 + 4O_{L^\infty}(1)) + O_{L^\infty}(\Delta t + \Delta x),
\end{align*}
 We remark here that we know that for subsonic (hyperbolic) systems $|v(t,x)|< \sqrt{\frac{K}{\rho_a}} < \infty$, and $\rho_{min} < \rho(t,x) < \rho_{max}$, for some $\rho_{min},\rho_{max} \in \mathbb{R}^+$. It is possible to find these values before the valve closure since we assumed the system is hyperbolic and the solution is in $C^{\infty}$. We then can write the $L^2$ norm of the first variation as follows
\begin{align*}
	& \|M(f)_{[1]}\|_{L^2} \le \|f\|_{L^2}\left[1 + K_{11}r + K_{12}r^2 + 2C_o^2 + O_{L^\infty}(\Delta t)\right],\\
	& \|M(f)_{[2]}\|_{L^2} \le \|f\|_{L^2}\left[1 + K_{21}r + K_{22}r^2 + K_{23}\Courant^2 + O_{L^\infty}(\Delta t + \Delta x)\right],
\end{align*}
where the coefficients $K_{11}$, $K_{12}$, $K_{21}$, $K_{22}$, $K_{23}$ are the upper bound numbers in terms of the $L$-infinity norm respectively. Since $\frac{\Delta t}{\Delta x} < \Courant^2$, we then have
\begin{align*}
	\|M\|_{L^2}\le 1 + O_{L^\infty}(\Courant^2),
\end{align*}
By homogenity of the norm in $\Delta x$, we can conclude the same bound but with the $\ell^2$ norm. 
In order for this the order terms to be small, we must have that $\Courant\leq \varepsilon(\|u\|_{L^{\infty}})$. Since this operator norm was computed for fixed $t$ we can apply the operator $M$ iteratively to give the desired bound $C$ which is conditional since it depends on the time $t\in [0,T]$. After application of Lemma \ref{lemmal2} we have proved the proposition since Lemma \ref{lemmal2} gives the $\ell_2$ stability of the operator.
\end{proof} 
Recall the following theorem:
\begin{theorem*}\nonumber [subset of Strang, \cite{Strang1964}]\label{convergethm}
Suppose that $\varphi$ is a consistent operator with order of accuracy $p$, and that its first variation is $\ell_2$-stable. Then if $u,A,\alpha$ and $\varphi$ are smooth inside cone dictated by the CFL condition
\begin{equation*}
	U(t_0,x_0,\Delta t) = u(t_0,x_0) + O((\Delta t)^p).
\end{equation*}
\end{theorem*}
\begin{proof}[Proof of Theorem \ref{main}]
Using the above theorem from Strang \cite{Strang1964}, Lemma \ref{lemmal2}, Propositions \ref{main2prop} and \ref{mainprop} allow us to conclude the hypotheses hold with $p=2$ and then we conclude with our main result before the valve closure.
\end{proof} 
The reason we chose this method is that consistency and stability implies convergence for the semi-linear problem is not known for other methods in a straightforward way. 

\section{Discussion of the regularity of the solution $w$ of \eqref{pdep2}}\label{regularitydiscussion}

There is no generic existence and uniqueness result for the solution $w$ which solves \eqref{pdep1}. Therefore the statement of the main theorem needs some more explaining. We let $w_1=(\rho_1,q_1)$ solve \eqref{pdep1}. Let $H_x$ be the heavy side function which is $1$ on $(-\infty,L)$ and zero elsewhere. We let $w_2$ be the solution to the problem 
\begin{align}\label{eqhi}
&(\partial _t\rho + \partial_x q)= 0,\\ \nonumber 
&\partial_tq + \partial_x\left(\frac{q^2}{\rho} + p\right)=-\frac{c_fq|q|}{2D\rho},\\ \nonumber
&\lim\limits_{t\rightarrow t_1^-}(\partial_t\rho_1(t,x)H_x,q_1(t,x)H_x)=(\partial_t\rho_2(t_1,x),q_2(t_1,x))\\ \nonumber
&\rho(t,0) = \rho_1(t) \quad t\in[0,T] \quad q(t,L)=0 \quad t\in [t_1,T]. 
\end{align}
We let $w$ be piecewise defined
\begin{equation*}
	w = \begin{cases}
	w_1\ \ t <t_1\\
	w_2 \ \ t\geq t_1.
	\end{cases}
\end{equation*}
This proposed solution should solve \eqref{pdep2} in the required sense. In this case the heaviside function $H(\Omega)$ is redundant. Notice that we have not made any claims on the regularity of the solution through the point $t_1$. It is not known if such a solution $w$ exists because evolution of the PDE in \eqref{pdep2} depends on the solution itself. The hyperbolic system of type \eqref{pdep1} has been show to have regularity $H^{s}((-\infty,L))$ with $s>1/2$ whenever the initial data has the same regularity, even when a Dirichlet boundary condition is imposed at $x=L$, \cite{BealsMetivier1986}.  For hyperbolic systems with solution independent coefficients, low regularity data including heaviside functions has been examined in \cite{metivierlow}. There is a gap in the literature for this type of system, although ad-hoc solutions with discontinuous data can be constructed. 

Therefore if the solution does not exist, we consider a smoothed out version of $H_x$, say $H_{\epsilon}$, and use the method of reflections, to instead construct an approximate solution to instantaneous valve closure as our desired modelling target. We know that the system \eqref{pdep2} has a form in Riemann coordinates which is equivalent to a coupled system of quasilinear transport equations \cite{BastinCoron2016book}, Example 1.6.1. To see then why this idea is a realistic model for $w_2$, we consider the following examples. First we consider
\begin{example}
\begin{equation}\label{exmoc1}
\begin{cases}
    \partial_tu + c(u)\partial_xu = 0 \quad  c(u) > 0\\
    u_0(x) = f(x)H(x) \\
    u(t,0) = 0
\end{cases} (t,x)\in [0,T]\times(-\infty,+\infty),
\end{equation}
where $f \in C^1(\mathbb{R})$ and $H$ is Heaviside function.
\end{example}
Then we can solve this transport equation by the method of characteristics
\begin{align*}
    \frac{dx}{dt} = c(u), \quad
    \frac{du}{dt} = 0
\end{align*}
which gives the solution by an implicit form
\begin{equation}
    u(t,x) = f(x-c(u)t)H(x-c(u)t),
\end{equation}
and by the reflection principle, the wave is symmetric around $x = 0$ so we construct it as 
\begin{equation}\label{sollinearex}
    u(t,x) = f(x-c(u)t)H(x-c(u)t) - f(-x-c(u)t)H(-x-c(u)t).
\end{equation}
Generally the solution \eqref{sollinearex} is not guaranteed to exist with the discontinuous initial data, except for some cases where the wave speed $c(u)$ is a polynomial in $u$, e.g. $c(u) = u$. However, if we replace the Heaviside function $H$ with a smoothed one, (say $H_\epsilon \in C^2$), then solution of the smoothed problem exists at least for a short period of time, which can be presented as
\begin{equation}
    u(t,x) = f(x-c(u)t)H_\epsilon(-x-c(u)t) - f(-x-c(u)t)H_\epsilon(-x-c(u)t).
\end{equation}
Similarly, if we cast the transport equation in a nonlinear form
\begin{example}
\begin{equation}\label{exmoc2}
\begin{cases}
    \partial_tu + c(u)\partial_xu = -b|u|^2, \quad b>0 \quad c(u) > 0\\
    u_0(x) = f(x)H_\epsilon(x) \\
    u(t,0) = 0
\end{cases}(t,x)\in [0,T]\times(-\infty,+\infty).
\end{equation}
\end{example}
Then the approximated solution by reflection is given by
\begin{align}
     u(t,x) = \frac{f(x-c(u)t)H_\epsilon(x-c(u)t)}{btf(x-c(u)t)H_\epsilon(x-c(u)t) + 1} + \frac{f(-x-c(u)t)H_\epsilon(-x-c(u)t)}{btf(x-c(u)t)H_\epsilon(-x-c(u)t) - 1}.
\end{align}
This approximate solution by the method of characteristics is around 0 at $x= 0$ with small $b$, and is depending continuously on the initial data, and will be valid before a critical time $t^*$ when the characteristic curves cross. Approaching this time $t^*$, the solution graph is becoming vertical somewhere in $x$ as
\begin{equation}
    \frac{\partial u}{\partial x}(t,x) \to \infty \quad \text{when}\quad  t \to t^* .
\end{equation}
This critical time $t^*$ can be determined by the implicit formula as
\begin{align}
    \frac{\partial u}{\partial x} = \frac{\partial}{\partial x}\left(\frac{f(\xi)H_\epsilon(\xi)}{btf(\xi)H_\epsilon(\xi)+1}\right),\quad \text{where}\  \xi = x - c(u)t.
\end{align}
Re-arranging the equation above we can have that
\begin{equation}
    \frac{\partial u}{\partial x} = \frac{[f(\xi)H_\epsilon(\xi)]'}{(bf(\xi)H_\epsilon(\xi)t + 1)^2 + [f(\xi)H_\epsilon(\xi)]'c'(u)t}.
\end{equation}
If there exists an earliest time $t^*$ that making the denominator in the right hand side of the above equation be 0, then the characteristic curves cross. When $c(u)>0$ is linear in $u$ $t^*$ can be quite large. Since $H_{\epsilon}$ can be constructed to approximate $H_x$ in $L^2$ norm so the approximate solution converges to the weak one, if it it exists when $b=0$. If $b\neq 0$, then the boundary conditions are less well preserved as time evolves. 

Similarly, a lengthy computation in Riemann coordinates shows the approximate solution $\tilde{w}_2$ consisting of two $C^2$ reflected initial conditions to the the system \eqref{pdep2}, after the valve closure is in $C^2$ for times depending on the norm of the initial data and is accurate for the nonlinear problem when either $c_f$ or $t$ is small. In this case $H(\Omega)$ restricts $\tilde{w}_2$ to the left of the valve. The numerics simulate the reflected approximate solution $\tilde{w}_2$ after the valve closure. There are no known results on the numerical convergence of the method of characteristics for the system with $c_f\neq 0$ which is why we chose the Lax-Wendroff scheme instead. 

\section{Switched boundary conditions}\label{bcS}
In the previous sections we discussed the consistency and convergence of the nonlinear Lax-Wendroff scheme through the finite difference operator (\ref{fds}). However, the boundary conditions are not specified, and therefore we have to limit the pipe in a certain range so that we can apply the boundary conditions to the scheme and see how it interact with the fluid flow inside the pipe. For simplicity, we will focus on the interacting part of length $2L$, so that we apply the numerical scheme on $\Omega$ as for problem \eqref{pdep2} described in the introduction. To show the scheme is convergent with dynamical boundary conditions we will do it in 2 cases, which are before and after the valve closure, each which introduces a different set of boundary conditions. 

Similar to the numerical scheme setup section, we need to discretize the pipe. We set $\Delta t$ and $\Delta x$ as our smallest units for computation of time and space respectively. We let $N_{\Delta x} = \frac{L}{\Delta x}$ be an integer for convenience, and then on a $2L$-long pipe we have $2N_{\Delta x} + 1$ points.\\

\subsection{Case 1: before the valve closure}\label{BCbefore}
Before the valve closure, the boundary conditions are given by the pressure at the 2 ends of the pipe, which is related to the density of the water by the pressure law and for the convenience of computations we just assume $\rho_1(t)$ and $\rho_2(t)$ are 2 constants. The initial conditions are also known beforehand and given by 
\begin{align*}
\rho(0,0), \rho(0,\Delta x), \rho(0,2\Delta x),..., \rho(0, N_{x}\Delta x)\\
v(0,0), v(0,\Delta x), v(0,2\Delta x),..., v(0, N_{x}\Delta x).
\end{align*}
Recall that we are using central scheme for the difference operator, every time iteration we will lose 2 points at the boundary side, and notice that the boundary conditions of the density $\rho$ are constant, therefore we need to find them of the velocity. \\

Recall again the PDE
\begin{align*}
\partial_t\left(\begin{matrix}
\rho\\v
\end{matrix}\right) + \left(\begin{matrix}
v&\rho\\\frac{K}{\rho_a\rho}&v
\end{matrix}\right)\partial_x\left(\begin{matrix}
\rho\\v
\end{matrix}\right) = \left(\begin{matrix}
0\\-\frac{c_f}{2D}v|v|
\end{matrix}\right).
\end{align*}
Notice that in the mass balance of the PDEs above we have at the boundary side of $x = 0$ 
\begin{align*}
0 = \partial_t\rho(t,x)|_{x = 0} = -\rho(t,0)\partial_xv(t,0) - v(t,0)\partial_x\rho(t,0),
\end{align*}
which motivates us to impose for $P$ and $V$ (discretized from $\rho$ and $v$ respectively) the condition 
\begin{align}\label{bdy1}
P(t,0)\frac{V(t,\Delta x) - V(t,0)}{\Delta x} + V(t,0)\frac{P(t,\Delta x) - P(t,0)}{\Delta x} = 0
\end{align}
Replacing $t$ with $t + \Delta t$ we have the new iteration for $V$ at boundary side $x = 0$:  
\begin{align*}
	V(t + \Delta t,0) = \frac{P(t+ \Delta t,0)V(t+ \Delta t,\Delta x)}{2P(t+ \Delta t,0) - P(t+ \Delta t,\Delta x)}.
\end{align*}
Same with the other side of the pipe, we have for $x = 2L$ that
\begin{align*}
P(t,2L)\frac{V(t,2L) - V(t,2L - \Delta x)}{\Delta x} + V(t,2L)\frac{P(t,2L) - P(t,2L - \Delta x)}{\Delta x} = 0,
\end{align*}
and the result is that
\begin{align*}
	V(t + \Delta t,2L) = \frac{P(t+ \Delta t,2L)V(t+ \Delta t,2L-\Delta x)}{2P(t+ \Delta t,2L) - P(t+ \Delta t,2L - \Delta x)}.
\end{align*}
\subsection{Case 2: after the valve closure}
When the valve is suddenly closed at $t = t_1$ at the position $x = L$, the boundary conditions will switched accordingly as we stated in (\ref{pdep2}). The initial conditions for $\rho$ and $v$ can be computed from the case 1 when setting the time $t = t_1$.\\

For the boundary conditions, the left side which connects to a reservoir does not change at all, correspondingly we have that 
\begin{align}
V(t + \Delta t,0) = \frac{P(t+ \Delta t,0)V(t+ \Delta t,\Delta x)}{2P(t+ \Delta t,0) - P(t+ \Delta t,\Delta x)}.
\end{align}

At the right side $x = L$, the valve is now closed, so we have $v(t,L) = 0$ for $t > t_1$, and we need to calculate the density (pressure). We are using the momentum balance at $x = L$ so that
\begin{align}\label{bdy2}
v(t,L)\partial_xv(t,L) + \frac{1}{\rho(t,L)}\frac{K}{\rho_a}\partial_x\rho + \frac{c_f}{2D}v(t,L)|v(t,L)| = 0. 
\end{align}
Since $v(t,L) = 0$, and $\rho(t,L)$ can not be 0, we then have $\partial_x\rho(t,L) = 0$ which leads to
\begin{align*}
P(t,L) = P(t, L - \Delta x).
\end{align*}
Notice that when computing the boundary conditions by numerical method, we are using forward and backward scheme, which is different from the central scheme we used for the operator $\varphi$. Therefore, we have to make sure that there will be no lower order (less than 2) error propagating in the iterations. So we have the following corollary: 
\begin{cor}\label{thmbdy}
	The finite difference operator $\varphi$ of Lax-Wendroff scheme (\ref{fds}) for the quasi-linear hyperbolic equation (\ref{hivp}) with dynamical boundary conditions is consistent of order 2.
\end{cor}
\begin{proof}[Proof of Corollary \ref{thmbdy}]
	We have already proved the operator $\varphi$ is consistent and $\ell_2$-stable on an infinite pipe or inside a pipe in previous sections. The next step is to verify that, given boundary conditions whether the operator $\varphi$ stays consistent, i.e., the errors imported by the dynamical boundary conditions are higher or equal to order 2.
\begin{enumerate}
\item Before the valve closure, we have to compute the boundary values for $v(t,0)$ and $v(t,L)$ such that from (\ref{bdy1})  which we have
\begin{align}
&\rho(t,0)\left(\frac{v(t,\Delta x) - v(t,0)}{\Delta x} + O((\Delta x)^2)\right) \nonumber
+ v(t,0)\left(\frac{\rho(t,\Delta x) - \rho(t,0)}{\Delta x} + O((\Delta x)^2)\right) = 0,
\end{align}
which gives
\begin{align*}
&v(t,0)= \frac{\rho(t,0)v(t,\Delta x) + O((\Delta x)^2)}{2\rho(t,0) - \rho(t,\Delta x) + O((\Delta x)^2)}=\\& \frac{\rho(t,0)v(t,\Delta x)}{2\rho(t,0) - \rho(t,\Delta x) + O((\Delta x))^2} + O((\Delta x)^2).
\end{align*}
Set $M_1 = \rho(t,0)v(t,\Delta x)$ and $M_2 = 2\rho(t,0) - \rho(t,\Delta x)$ and then we have
\begin{align*}
v(t,0) = \frac{M_1}{M_2 + O((\Delta x)^2)} = \frac{M_1}{M_2}\cdot\frac{1}{1 + O((\Delta x)^2)/M_2}.
\end{align*}
By Taylor's expanding the term $\frac{1}{1 + (O((\Delta x)^2))/M_2}$ we have that
\begin{align*}
&v(t,0) = \frac{M_1}{M_2}\left(1 - \frac{O((\Delta x)^2)}{M_2} + ...\right)= \frac{M_1}{M_2} + O((\Delta x)^2)\\
&= \frac{\rho(t,0)v(t,\Delta x)}{2\rho(t,0) - \rho(t,\Delta x)} + O((\Delta x)^2).
\end{align*}
These values are known explicitly because they are given to us by the reservoir condition and the initial conditions on the PDE. On the other side of the pipe where $x = L$ we use backward scheme for the spatial derivative and with the same computations we have that for $x = L$:
\begin{align}\label{p2bc1}
v(t,L) = \frac{\rho(t,L)v(t,L - \Delta x)}{2\rho(t,L) - \rho(t,L - \Delta x)} + O((\Delta x)^2).
\end{align}
We replace the numerical iterations values with their values $(\rho,v)$ to $(P,V)$. 

\item After the valve closure, the computation for $v(t,0)$ remains the same as it is in 1), but at $x = L$, $v(t,L)$ is set to 0 as the boundary condition changed. Then we need to compute $\rho(t,L)$ for the completion of the iterations. We have from (\ref{bdy2}) that
\begin{align}\label{p2bc2}
&v(t,L)\partial_xv(t,L) + \frac{1}{\rho(t,L)}\frac{K}{\rho_a}\partial_x\rho + Cv(t,L)|v(t,L)| = 0\\ \nonumber
%\Rightarrow &  \frac{K}{\rho_a\rho(t,L)}\left(\frac{\rho(t,L - \Delta x) - \rho(t,L)}{\Delta x} + O(\Delta x)\right) = 0\\ \nonumber
%\Rightarrow & \rho(t,L) = \rho(t,L - \Delta x) + O((\Delta x)^2).
\end{align}
\end{enumerate}
We update the boundary conditions when the cutoff is introduced to be order two matching the solution which is now reflected off the valve by the extra condition of the boundary conditions.  Indeed, after the valve closure the PDE gives $\partial_x\rho(t,L)=0$. In order to also solve the PDE $\rho$ must be locally constant. We then have that 
\begin{align}
\rho(t,L)=\rho(t,L-\Delta x)
\end{align}
But $\rho(t,L)=P(t,L)+O(\Delta x)^2$ before the valve closure. This allows us to import the errors as only being of order 2 because the process starts anew. 
\end{proof}

\section{Weak convergence with discontinuous initial conditions}\label{weakS}

Until now, we have proved the convergence between the numerical solution of the Lax-Wendroff scheme and the analytical solution of the water hammer problem when the boundary conditions are constant or continuous and without a sudden change (the sudden closure of the valve), i.e., before the valve closure. However, no existing research has shown that the method proposed in \cite{Strang1964} of first variance still works when the boundary conditions are discontinuous (including jumps) and therefore the analytical solution is not smooth either. In this case, we will refer to the weak solutions of the nonlinear hyperbolic conservation law. Meanwhile, the finite difference scheme (Lax-Wendroff) is still valid after the valve closure, such that it simulates the flow state afterwards until the flow achieves a new equilibrium. In this case, we can still possibly have $weak\ convergence$ between the weak solution of the PDE and the numerical solutions of Lax-Wendroff scheme. In other words we have that  that the piecewise constant solution $u^n_j$ from the set of numerical solutions $\{u_j^n\}_{n,j \in \mathbb{N}}$, converges $u_n^j \rightharpoonup u(t,x)$ weakly as $\Delta x, \Delta t \rightarrow 0$. We did not use another numerical method because we did not see how to incorporate these moving boundary conditions otherwise.

Since we have defined the $\mathbb{H}(\Omega)$ function, we can restrict to the region that we are interested in, which is $\Omega_1 := [t_1,T]\times[0,L]$ as the region set which has not already been covered by the previous cases. Let $\mathbb{H}_1:\mathbb{R}^2 \rightarrow \{0,1\}^{2\times2}$ be such that 
\begin{align*}
\mathbb{H}_1(t,x):=\left(\begin{matrix}
h_{11}&h_{12}\\h_{21}&h_{22}
\end{matrix}\right)=\big[H(t - t_1) - H(t - T)\big]\big[H(x) - H(x - L)\big]\cdot \textbf{I}_{2\times2}
\end{align*}
where the Heaviside function $H(\cdot) : \mathbb{R}\rightarrow \{0,1\}$ is defined as
\begin{equation*}
	H(s) = \begin{cases}
	1\ \ s \ge 0\\
	0 \ \ s < 0.
	\end{cases}
\end{equation*}
This $\mathbb{H}_1$ is supported on $[t_1,T]\times [0,L]$ is used to examine the dynamics after the valve closure. Let $\varphi(t,x) := [\varphi_1(t,x) \ \   \varphi_2(t,x)]^\top$, with $\varphi_i(t,x) \in C_0^\infty(\Omega_0), i = 1,2$ be any test functions. We will focus on the domain $\Omega_1$ which is after the valve closure for the following analysis. Using the above $\mathbb{H}_1$ and letting $w\in C^{\infty}([0,t_1)\times [0,2L])^2\cup H^{-1}([t_1,T]\times [0,L])^2$ denote the solution to \eqref{pdep2} for the rest of this section we obtain
\begin{equation*}
\iint\limits_{\Omega_0}\mathbb{H}_1\left(\partial_tw + \partial_xF(w) - G(w)\right)\cdot \varphi \ dxdt = 0.
\end{equation*}
In order to perform integral by parts on the equation above, we need to make sure that the functions or distributions are well defined, and here we consider a sequence of smooth functions $\mathbb{H}_{1,\epsilon}: \mathbb{R}^2\to \{0,1\}^2$ that converges to $\mathbb{H}_1$ in $L^2$ space, as $\epsilon \to 0$, where $\mathbb{H}_{1,\epsilon}$ is defined as
\begin{equation}
\mathbb{H}_{1,\epsilon} := \mathbb{H}_{1} \star \chi_\epsilon, \quad  \star \text{ is the convolution operator,}
\end{equation}
with 
\begin{equation}
\chi_\epsilon:= 
\begin{cases}
\frac{1}{c \epsilon^2}\exp{\left(-\frac{1}{1-|\xi/\epsilon|}\right)} \quad |\xi| < 1\\
0 \quad |\xi| \ge 1
\end{cases},
\end{equation}
and
\begin{equation*}
c = \int\limits_{|\xi| < 1}  \exp{\left(-\frac{1}{1-|\xi|}\right)} d\xi, \quad |\xi| = \sqrt{t^2 + x^2}.
\end{equation*}
So that as $\epsilon \to 0$, we have 
\begin{equation}
\|\mathbb{H}_{1,\epsilon} - \mathbb{H}_1\|_{L^2(\Omega_1)} \to 0.
\end{equation}
Therefore we consider the following integral
\begin{equation}\label{weakform1}
\iint\limits_{\Omega_0}\left(\partial_tw + \partial_xF(w) - G(w)\right)\cdot  (\mathbb{H}_{1,\epsilon}\varphi) \ dxdt, \epsilon > 0
\end{equation}
Integrating by parts, \eqref{weakform1} is equal to
\begin{align*}
&\int\limits_{t_1}^T\left([\mathbb{H}_{1,\epsilon}F(w)\cdot\varphi]^{L+1}_{-1} - \int\limits_{-1}^{L+1}F(w)\cdot \partial_x(\mathbb{H}_{1,\epsilon}\varphi)\ dx\right)dt \\
&+ \int\limits_{-1}^{L+1}\left([\mathbb{H}_{1,\epsilon}w\cdot\varphi]^T_{t_1} - \int\limits_{t_1}^T w\cdot\partial_t(\mathbb{H}_{1,\epsilon}\varphi)dt\right)dx\\
&= \int\limits_{t_1}^T\left([\mathbb{H}_{1,\epsilon}F(w)\cdot\varphi]^{L+1}_{-1} - \int\limits_{-1}^{L+1}F(w)\cdot (\partial_x\mathbb{H}_{1,\epsilon})\varphi + F(w)\cdot(\mathbb{H}_{1,\epsilon}\partial_x\varphi)\ dx\right)dt\\
& + \int\limits_{-1}^{L+1}\left([\mathbb{H}_{1,\epsilon}w\cdot\varphi]^T_{t_1} - \int\limits_{t_1}^T w\cdot(\partial_t\mathbb H_{1,\epsilon})\varphi + u\cdot (\mathbb H_{1,\epsilon}\partial_t\varphi)dt\right)dx\\
& = -\iint\limits_{\Omega_1}w\cdot\partial_t\varphi + F(w)\cdot\partial_x\varphi\ dxdt - \int\limits_{-1}^{L+1}w(t_1,x)\cdot\varphi(t_1,x)dx\\
&- \iint\limits_{\Omega_0}w\cdot(\partial_t\mathbb{H}_{1,\epsilon}\varphi) + F(w)\cdot(\partial_x\mathbb{H}_{1,\epsilon}\varphi)\ dxdt,
\end{align*}
where we have that
\begin{align*}
\iint\limits_{\Omega_0}w\cdot(\partial_t\mathbb{H}_{1,\epsilon}\varphi) + F(w)\cdot(\partial_x\mathbb{H}_{1,\epsilon}\varphi)\ dxdt = 
\iint\limits_{\Omega_0}(w\partial_t\mathbb{H}_{1,\epsilon})\cdot\varphi + (F(w)\partial_x\mathbb{H}_{1,\epsilon})\cdot\varphi\ dxdt
\end{align*}
By the construction of the indicator function and taking $\epsilon \to 0$, at $(t,x) = (t_1,L)$ we have 
\begin{align*}
\lim_{\epsilon \to 0}\partial_t \mathbb{H}_{1,\epsilon} = \partial_t\mathbb{H}_1&= \big[H(x) - H(x - L)\big]\delta(t - t_1)\cdot\mathbf{1}_{2\times2}\\ 
\lim_{\epsilon \to 0}\partial_x \mathbb{H}_{1,\epsilon} = \partial_x\mathbb{H}_1&= -\big[H(t - t_1) - H(t - T)\big]\delta(x - L)\cdot\mathbf{1}_{2\times2}.
\end{align*}
Then we have
\begin{align*}
\iint\limits_{\Omega_0}(w\partial_t\mathbb{H}_1)\cdot\varphi + (F(w)\partial_x\mathbb{H}_1)\cdot\varphi\ dxdt = 
\int\limits_0^Lw(t_1,x)\cdot\varphi(t_1,x)dx + \int\limits_{t_1}^TF(w(t,L))\varphi(t,L)dt.
\end{align*}
Applying the initial and boundary conditions, the solution $w$ therefore satisfies the integral formulation, which is the same form as in Theorem \ref{main2}:
\begin{align}\label{weakform2}
	\iint\limits_{\Omega_0}w\cdot\partial_t\varphi + F(w)\cdot\partial_x\varphi\  &- G(w)\cdot \varphi \ dxdt + \int\limits_0^{L}w(t_1,x)\cdot\varphi(t_1,x)dx \\ \nonumber
	& +\int\limits_0^Lw(t_1,x)\cdot\varphi(t_1,x)dx - \int\limits_{t_1}^TF(w(t,L))\varphi(t,L)dt = 0.
\end{align}
Note that the solution $w$ is not smooth due to the sudden change of the boundary condition at $x = L$ at time $t = t_1$, whence the assumption $w\in H^{-1}([t_1,T]\times [0,L])^2$.  \\

Recall that $\Delta t$ and $\Delta x$ are minimal time and space unit of discretization, and then we have a series of evolution points for the FDM inside the domain $\Omega$ such that
\begin{align*}
	(n\Delta t, j\Delta x), \ \ n,j = 0,1,2,3,...
\end{align*}
In total, we have $(N_x + 1)\times (N_t + 1)$ points, where $N_x := \frac{L}{\Delta x}$ and $N_t := \frac{T}{\Delta t}$. On top of those discrete points, suppose we have a set of (normalized) numerical solutions $\{U_j^n\}_{(n,j)\in I}$ which are computed from the Lax-Wendroff scheme \eqref{fds}, such that the corresponding numerical solution which is not normalized, $W_j^n=(\rho_j^n, q_j^n)$, satisfies
\begin{equation}\label{weaklw}
	\partial_{t,\Delta t}W + \frac{F(W(t,x + \Delta x)) - F(W(t,x-\Delta x))}{2\Delta x} = G(W),
\end{equation}
where $\partial_{t,\Delta t}W$ in \eqref{weaklw} is approximated by first and second order space derivatives so that
\begin{align*}
	\partial_{t,\Delta t}W :&= \frac{W(t + \Delta t) - W(t)}{\Delta t} - \frac{\Delta t}{2}\left(\partial_x(\partial_wF)\cdot\partial_xF  + \partial_wF\partial_x^2F + \partial_x(\partial_wF\cdot G)\right) + O(\Delta t^2).
\end{align*}
Now we proceed to the proof of the second main theorem. We first recall the following elementary definition
\begin{defn}
The order of a distribution $f:\mathcal{D}(\Omega)\rightarrow \mathbb{R}$ with $\Omega \subset \mathbb{R}^2 $ an open domain is the smallest integer $N\in\mathbb{N}$ independent of $\Omega$ such that the following inequality holds
\begin{align}
\left|\int\limits_{\Omega}f(\varphi)\,dxdt\right|\leq C_{\Omega}\|\varphi\|_{\Omega,N} \quad \varphi\in C_c^{\infty}(\Omega)
\end{align}
with some constant $C_{\Omega}$ depending on $\Omega$ and 
\begin{align}
\|\varphi\|_{\Omega,N}=\sup\limits_{\substack {\alpha\leq N \\ (t,x)\in\Omega}}|D^{\alpha}\varphi|. 
\end{align}
\end{defn}
The norm of the test functions denoted above will show up in the proof and help us prove Theorem \ref{main2}. 
\begin{proof}[Proof of Theorem \ref{main2}]
	For the convenience of computation and without the loss of generality, we reset $t_1 = 0$, accordingly $T$ becomes $T-t_1$. Let us first discretize the domain of $\Omega_0$, such that it is combined with every unit domain $\Omega_j^n$:
	\begin{align}
		\Omega_0 := \bigcup_{n = 1}^{N_t}\bigcup_{j = 1}^{N_x}\Omega_j^n,
	\end{align}
	where each unit of $\Omega_j^n$ is defined by
	\begin{align}
		\Omega_j^n := \big[(n-1)\Delta t,n\Delta t\big]\times\big[(j - 1)\Delta x,j\Delta x\big].
	\end{align}
	Then we can extend $W_j^n$ to be piecewise constant functions as $W_j^n(t,x)$, on each unit of domain $\Omega_j^n$ with
	\begin{equation}\label{extendu}
		W^n_j(t,x) = 
		\begin{cases}
			W^n_j, \ \ (t,x)\in \Omega^n_j\\
			0, \ \ \ \text{otherwise}.
		\end{cases}
	\end{equation}
	We let $\{\varphi_j^n(t,x)\}_{n,j\in \mathbb{N}}$ be a set of test functions that defined on each unit of domain $\Omega(\varphi)^n_j$ such that
 \begin{equation}\label{testfunc}
		\varphi^n_j(t,x) := \frac{1}{5\Delta t\Delta x}\iint\limits_{\Omega(\varphi)_j^n}\varphi(t,x)\ dxdt, 
	\end{equation}
 where
	\begin{align}
		\Omega(\varphi)^n_j:= \Omega_{j-1}^n\cup\Omega_j^{n-1}\cup\Omega_j^n\cup\Omega_j^{n+1}\cup\Omega_{j+1}^n.
	\end{align}
	
	To prove convergence, we define $\tilde{\varphi}(t,x)$ and $\tilde{W}(t,x)$ such that for any $(t,x)\in\Omega_0$ we have,
	\begin{align}\label{discsum1}
		\tilde{\varphi}(t,x) &:= \sum_{n = 1}^{N_t}\sum_{j=1}^{N_x}\varphi^n_j(t,x),
	\end{align}
		\begin{align}\label{discsum2}
		\tilde{W}(t,x) &:= \sum_{n = 1}^{N_t}\sum_{j=1}^{N_x}W^n_j(t,x),
	\end{align}
	so that the 2 piece-wise constant functions above can approximate the distributional solution $w(t,x)$ and the test function $\varphi(t,x)$. Accordingly, the indicator function $\mathbb{H}_1$ also needs to be discretized, we set
	\begin{equation}\label{ind}
		\tilde{\mathbb{H}}(t,x) := \sum_{n = 1}^{N_t}\sum_{j=1}^{N_x}\mathbb{H}^n_j(t,x) = \left(\sum_{n = 1}^{N_t}\sum_{j=1}^{N_x}\mathbf{1}_{\Omega_j^n}(t,x) \right)\cdot\textbf{I}_{2\times2},
	\end{equation}
	where
	\begin{equation*}
		\mathbf{1}_{\Omega_j^n}(t,x):=\begin{cases}
		1,\ \ (t,x)\in\Omega_j^n\\
		0,\ \ \ \text{otherwise}.
		\end{cases}
	\end{equation*}
	Putting the discrete sum \eqref{discsum1} and \eqref{discsum2} into the integral (weak) form to get
	\begin{align}\label{discweakform}
		I(\tilde{W}):=\iint\limits_{\Omega_1}\tilde{W}\cdot\partial_t\tilde{\varphi} &+ F(\tilde{W})\cdot\partial_x\tilde{\varphi} - G(\tilde{W})\cdot\tilde{\varphi} \ dxdt \\ \nonumber
		&+ 2\int\limits_0^L\tilde{W}(t_1,x)\cdot\tilde{\varphi}(t_1,x)dx - \int\limits_{t_1}^TF(\tilde{W}(t,L))\cdot\tilde{\varphi}(t,L)dt
	\end{align} 
	We say the discrete sum function $\tilde{W}$ induced by the numerical solution set $\{W_j^n\}_{n,j\in \mathbb{N}}$, weakly converges to the distributional solution $w$ if the integral $I(\tilde{W})$ in \eqref{discweakform} approaches 0 when $\Delta x, \Delta t \to 0$ (which means $N_x,N_t \rightarrow \infty$, but $N_x\Delta x = L,N_t\Delta t = T$) for any test function $\varphi$, denoted by $\tilde{W}(t,x) \rightharpoonup w(t,x)$ for any $(t,x) \in \Omega_0$. The following is how we compute term by term for $I(\tilde{W})$.
	
	For convenience, we also let $F(W^n_j)$ be short for $F(W^n_j(t,x))$, and let $G(W^n_j)$ be short for $G(W^n_j(t,x))$. Similar to the method from \cite{PerthameSimeoni2003}, we multiply the equation \eqref{weaklw} by $\varphi_j^n(t,x)$ and sum over $n$ and $j$, to obtain 
	\begin{align}\label{weaksum}
	\sum_{n = 1}^{N_t}\sum_{j = 1}^{N_x}\mathbb{H}^n_j\left(\frac{F(W_{j+1}^n) - F(W_{j-1}^n)}{2\Delta x}  + \partial_{t,\Delta t}W  - G(W_j^n)\right)\cdot\varphi_j^n\Delta x\Delta t = 0 ,
	\end{align}
	where $\partial_{t,\Delta t}W$ is the same as in \eqref{weaklw}. Consider the first summation in \eqref{weaksum}, we have 
	\begin{align}
& S_x := \sum_{n = 1}^{N_t}\sum_{j = 1}^{N_x}\mathbb{H}^n_j\left(\frac{F(W_{j+1}^n) - F(W_{j-1}^n)}{2\Delta x}\right)\cdot\varphi_j^n\Delta x\Delta t=\underbrace{\sum_{n = 1}^{N_t}\sum_{j = 1}^{N_x-1}\frac{F(W_{j+1}^n) - F(W_{j-1}^n)}{2\Delta x}\cdot\varphi_j^n\mathbf{1}_{\Omega_j^n} \Delta x\Delta t}_{S_{x,1}} \\ \nonumber
	&+ \underbrace{\sum_{n = 1}^{N_t}\frac{F(W_{N_x+1}^n) - F(W_{N_x-1}^n)}{2\Delta x}\cdot\varphi_{N_x}^n\mathbf{1}_{\Omega_{N_x}^n}\Delta x\Delta t}_{S_{x,2}}
	\end{align}
by summation by parts, we have that
\begin{align*}
& S_{x,1}= \sum_{n = 1}^{N_t}\sum_{j = 1}^{N_x-1}(F(W_{j+1}^n) - F(W_{j-1}^n))\cdot\varphi_j^n\mathbf{1}_{\Omega_j^n}) \Delta t = \\&\sum_{n = 1}^{N_t}\bigg[\underbrace{F(W^n_2)\cdot\varphi^n_1 - F(W^n_{0})\cdot\varphi^n_1}_{j = 1} + \underbrace{F(W^n_3)\cdot\varphi^n_2 - F(W^n_1)\cdot\varphi^n_2}_{j = 2} + \\&\underbrace{F(W^n_{4})\cdot\varphi^n_3 - F(W^n_2)\cdot\varphi^n_3}_{j = 3} + ... + \underbrace{F(W^n_{N_x})\cdot\varphi^n_{N_x-1} - F(W^n_{N_x - 2})\cdot\varphi^n_{N_x-1}}_{j = N_x-1}\bigg]\mathbf{1}_{\Omega_j^n}\Delta t.
	\end{align*}
Now we re-pair and classify the summation with $F(W^n_j)$ instead of $\varphi^n_j$ in order to construct derivatives of $\varphi^n_j$, which is similar to integration by parts 
	\begin{align*}
		&S_{x,1}= \sum_{n = 1}^{N_t}\left(\frac{1}{2}\sum_{j = 1}^{N_x}F(W_j^n)\cdot(\varphi_{j - 1}^n - \varphi_{j + 1}^n)\right)\Delta t \\ \nonumber
		&+ \frac{1}{2}\sum_{n = 1}^{N_t}(F(W_{N_x}^n)\cdot\varphi_{N_x - 1}^n + F(W_{N_x - 1}^n)\cdot\varphi_{N_x}^n - F(W_0^n)\cdot\varphi_{1}^n - F(W_{1}^n)\cdot\varphi_0^n)\mathbf{1}_{\Omega_j^n}\Delta t\\ 
		& = \sum_{n = 1}^{N_t}\left(\sum_{j = 1}^{N_x}F(W_j^n)\cdot\frac{(\varphi_{j - 1}^n - \varphi_{j + 1}^n)}{2\Delta x}\Delta x\right)\Delta t +\frac{1}{2}\sum_{n = 1}^{N_t}F(W_{N_x - 1}^n)\cdot\varphi_{N_x}^n\Delta t + \frac{1}{2}\sum_{n = 1}^{N_t}F(W_{N_x}^n)\cdot\varphi_{N_x-1}^n\Delta t.
	\end{align*}
	These 2 terms $F(W_0^n)\cdot\varphi_{1}^n $ and $F(W_{1}^n)\cdot\varphi_0^n$ will vanish using \eqref{testfunc}, \eqref{extendu} and \eqref{ind}. So we
	consider the rest of the terms and take the limit when $\Delta t,\Delta x \rightarrow 0$. Note that from the boundary condition of \eqref{pdep2}, we know due to the sudden closure of the valve, there is a discontinuity in both $\rho$ and $q$ at $(t,x) = (t_1,L)$, which leads to the discontinuity in $F(W(t,x))$ at $t_1$ and $L$ as well. Therefore, we need to exclude the term $F(W^1_{N_x - 1})$ because of the discontinuity introduced by closing the valve. So we have
	\begin{align*}
		&\lim_{\Delta t,\Delta x \to 0}\frac{1}{2}\sum_{n = 2}^{N_t}F(W_{N_x - 1}^n)\cdot\varphi_{N_x}^n\Delta t + \frac{1}{2}\sum_{n = 1}^{N_t}F(W_{N_x}^n)\cdot\varphi_{N_x-1}^n\Delta t=\\
		& \lim_{\Delta t,\Delta x\to 0}\frac{1}{2}\int\limits_{t_1 + \Delta t}^TF(\tilde W(t,L-\Delta x))\tilde \varphi(t,L)dt + \frac{1}{2}\int\limits_{t_1}^TF(\tilde{W}(t,L))\tilde \varphi(t,L-\Delta x)dt \\
		&= \int\limits_{t_1}^TF(w(t,L))\cdot\varphi(t,L)dt,
	\end{align*}
	and then for some positive constant $C_{\Omega_{N_x}^1,F}$ depending on $\Omega_{N^x}$ and the Lipschitz constant of $F$
	\begin{align*}
		\lim\limits_{\Delta t,\Delta x \to 0} F(W^1_{N_x - 1})\cdot\varphi^1_{N_x-1}\, \Delta t &= \lim_{\Delta t,\Delta x \to 0}\frac{F(W^1_{N_x - 1}) - F(W^1_{N_x + 1})}{2\Delta x} \cdot\varphi_{N_x}^1\mathbf{1}_{\Omega_{N_x}^1}\, \Delta x\Delta t\\
		& \le \lim_{\Delta t\to 0}\Delta t\cdot|\langle\partial_xF(\tilde{W}(t_1,x))|_{x = L},\tilde{\varphi}(t_1,x)\rangle|\\
		& = \lim_{\Delta t\to 0}\Delta t\cdot C_{\Omega_{N_x}^1,F}\cdot\|\varphi\|_{\Omega_{N_x}^1,1} = 0
	\end{align*}
Moreover we also have
	\begin{align*}
		&\lim_{\Delta t,\Delta x \to 0}S_{x,2} = \lim_{\Delta t,\Delta x \to 0}\sum_{n = 1}^{N_t}\frac{F(W_{N_x+1}^n) - F(W_{N_x-1}^n)}{2\Delta x}\cdot\varphi_{N_x}^n\mathbf{1}_{\Omega_{N_x}^n}\Delta x\Delta t = \\& \lim\limits_{\Delta x\to 0}\Delta x\cdot\int\limits^T_{t_1}\partial_xF(w(t,L))\cdot\varphi(t,L)dt
		 = \lim_{\Delta x\to 0} \Delta x\cdot F(w(t_1^-,L))\int\limits_{t_1}^T\delta(t - t_1)\varphi(t,L)dt =\\&  \lim_{\Delta x\to 0}\Delta xF(w(t_1,L))\varphi(t_1,L) = 0.
	\end{align*}
Combining the summations we obtain
	\begin{align*}
	\lim_{\Delta t,\Delta x \to 0}	S_{x}&=\lim_{\Delta t,\Delta x \to 0}S_{x,1} + \lim_{\Delta t,\Delta x \to 0}S_{x,2}= -\iint\limits_{\Omega_1}F(\tilde{W})\cdot\partial_x\tilde \varphi(t,x)dxdt + \int\limits_{t_1}^TF(\tilde W(t,L))\cdot\tilde \varphi(t,L)dt\\
	& = -\iint\limits_{\Omega_1}F(w)\cdot\partial_x\varphi(t,x)dxdt + \int\limits_{t_1}^TF(w(t,L))\cdot\varphi(t,L)dt.
	\end{align*}
Next, we compute for the summation over time part,
	\begin{equation}
		S_t = \sum_{n = 1}^{N_t}\sum_{j = 1}^{N_x}\mathbb{H}^n_j\partial_{t,\Delta t}W\cdot\varphi_j^n\Delta x\Delta t
	\end{equation}
	with
	\begin{equation}
	\partial_{t,\Delta t}W = \frac{W_j^{n + 1} - W_j^{n}}{\Delta t} - \frac{\Delta t}{2}\left(\partial_x(\partial_wF)\cdot\partial_xF + \partial_wF\partial_x^2F + \partial_x(\partial_wF\cdot G)\right) + O(\Delta t^2).
	\end{equation}
Then we have that the time derivative part of the sum is
	\begin{align*}
		S_t =  \underbrace {\sum_{n = 2}^{N_t}\sum_{j = 1}^{N_x}\partial_{t,\Delta t}W\cdot\varphi_j^n\mathbf{1}_{\Omega^n_j}\Delta x\Delta t}_{S_{t,1}} +  \underbrace{\sum_{j = 1}^{N_x}\partial_{t,\Delta t}W(0,x)\cdot\varphi_{N_x}^1\mathbf{1}_{\Omega^1_j}\Delta x\Delta t}_{S_{t,2}}.
	\end{align*}
	However at boundary side $x = L$, the flow is cut off, so the PDE becomes a linear homogeneous transport equation:
	\begin{equation}
		\mathbb{H}(t,L^-)(\partial_tw + \tilde A\partial_xw - G(w)) = 0,
	\end{equation}
	where 
	\begin{equation}
		\tilde A= \left(\begin{matrix}
		0&1\\\frac{K}{\rho_a} & 0
		\end{matrix}\right)
	\end{equation}	
	We have by Lax-Wendroff that
	\begin{align}
	\partial_tw(t,L) = \frac{w(t+\Delta t,L) - w(t,L)}{\Delta t} - \frac{\Delta t}{2}\tilde{A}^2\partial_x^2w(t,L) + O(\Delta t^2).
	\end{align}
	Now we have the summation for $\Delta x$ at $t \in [0,\Delta t]$ that
	\begin{align*}
	&S_{t,2} = \\& \sum_{j = 1}^{N_x-1}\left(\frac{W_{j}^{2} - W_{j}^{1}}{\Delta t} + O(\Delta t)\right)\cdot\varphi_{j}^1\mathbf{1}_{\Omega^1_j}\Delta x\Delta t - \tilde A^2\left(\frac{\Delta t}{2}\frac{W_{N_x + 1}^1 - 2W_{N_x}^1 + W_{N_x - 1}^1}{\Delta x^2}\right)\cdot\varphi_{N_x}^1\mathbf{1}_{\Omega^1_{N_x}}\Delta x\Delta t.
	\end{align*}
	Therefore we have 
	\begin{align*}
	\lim_{\Delta t,\Delta x \to 0}S_{t,2} &:= \lim_{\Delta t\to 0}\Delta t\cdot\int\limits_0^L\tilde{H}(t_1,x)\partial_t\tilde W(t_1,x)\cdot\tilde \varphi(t_1,x) dx + O((\Delta t)^2)\cdot\int\limits_0^L\tilde \varphi(t_1,x)dx\\
	&+ \lim_{\Delta t,\Delta x \to 0}\tilde A^2\left(\frac{\Delta t}{2}\frac{W_{N_x + 1}^1 - 2W_{N_x}^1 + W_{N_x - 1}^1}{\Delta x^2}\right)\cdot\varphi_{j}^1\Delta x\Delta t\\
	&= \int\limits_0^L\langle \partial_tw(t,x),\mathbb{H}_1(t,x)\varphi(t,x)\rangle\  dx + \lim_{\Delta t\to 0} \frac{\Delta t}{2}\langle\tilde{A}^2\partial_x^2w(t,L),\varphi(t,L)\rangle,
	\end{align*}
since \begin{align*}
\lim_{\Delta t\to 0} \frac{\Delta t}{2}|\langle\partial_x^2w(t,L),\varphi(t,L)\rangle| \le \lim_{\Delta t\to 0}\frac{\Delta t}{2}\cdot C_{\Omega_0,w}\|\varphi\|_{\Omega_0,2} = 0,
\end{align*}
where $C_{\Omega_0,w}$ is a constant which depends on $\Omega_0$ and $w$ so that 
\begin{align*}
\lim\limits_{\Delta t,\Delta x\rightarrow 0}S_{t,2} = -\int\limits_0^Lw(t_1,x)\cdot\varphi(t_1,x)dx.
\end{align*}

	%and accordingly the second summation 
	%\begin{align}
	%\lim_{\Delta t,\Delta x \rightarrow 0}S_{t,2,2} &:= \lim_{\Delta t,\Delta x \to 0}\sum_{j = 1}^{N_x}\left(\frac{\Delta t}{2}\frac{U_{j + 1}^1 - 2U_{j}^1 + U_{j - 1}^1}{\Delta x^2}\right)\cdot\varphi_{j}^1\chi^1_j\Delta x\Delta t,
	%\end{align}
	%with the CFL condition that $\frac{\Delta t}{\Delta x} \le C_{CFL}$ and IC and BC \eqref{icnbc} and $Theorem$ \ref{thmbdy}, then we have
	%\begin{align*}
	%\lim_{\Delta t,\Delta x \to 0}&S_{t,2,2} = \frac{1}{2}C_{CFL}\lim_{\Delta t,\Delta x \to 0}\sum_{j = 1}^{N_x-1}(U_{j + 1}^1 - 2U_{j}^1 + U_{j - 1}^1)\cdot\varphi_{j}^1\chi^n_{N_x}\Delta t\\
	%& = \frac{1}{2}C_{CFL}\lim_{\Delta t,\Delta x \to 0}\sum_{n = 1}^{N_x-1}\left(\frac {U^n_{N_x + 1} - U^n_{N_x}}{\Delta t} - \frac {U^1_{j} - U^1_{j}}{\Delta t}\right)\cdot\varphi_{N_x}^n\chi^n_{N_x}(\Delta t)^2\\
	%& = \lim_{\Delta t \rightarrow 0}\frac{\Delta t}{2}C_{CFL}\int_{t_1}^T(\partial_tu(t,L) - \partial_tu(t,L) + O(\Delta t))\cdot\varphi(t,L)dt\\
	%& = 0.
	%\end{align*}
	In the summation $S_{t,1}$, since we have excluded the case when $t = t_1$ where there is a discontinuity in $u$, then $S_{t,1}$ is a distribution of order 0 whenever $\Delta t$ or $\Delta x$ goes to zero, and with the assumption that $|U_j^n|\le C_h, \forall n,j$ and $F,G \in C^{\infty}(\Omega)$ then we have for $S_{t,1}$
	\begin{equation}
	S_{t,1} := \sum_{j = 1}^{N_x}\sum_{n = 2}^{N_t}\mathbb{H}^n_j\left(\frac{W^{n+1}_j - W^{n}_j}{\Delta t}\cdot\varphi_j^n+O(\Delta t)\right)\Delta x\Delta t
	\end{equation}
	and by summation by parts of $S_{t,1}$ we have
	\begin{align*}
&S_{t,1} = \sum_{j = 1}^{N_x}\sum_{n = 2}^{N_t}(W^{n+1}_j - W^{n}_j)\cdot\varphi_j^n\mathbf{1}_{\Omega^n_j}\Delta x + O(\Delta t)\cdot\text{Area}(\Omega_1)\\
		& = \sum_{j = 1}^{N_x}\sum_{n = 2}^{N_t}W^{n+1}_j\cdot(\varphi_j^n - \varphi^{n+1}_j)\Delta x+ \sum_{j = 1}^{N_x}-W^1_j\cdot\varphi^1_j\mathbf{1}_{\Omega^n_j}\Delta x + \sum_{j = 1}^{N_x}W^{N_t+1}_j\cdot\varphi^{N_t}_j\mathbf{1}_{\Omega^n_j}\Delta x+O(\Delta t)\cdot\text{Area}(\Omega_1)\\
		& = \sum_{j = 1}^{N_x}\sum_{n = 2}^{N_t}W^{n+1}_j\cdot\left(\frac{\varphi_j^n - \varphi^{n+1}_j}{\Delta t}\right)\Delta t\Delta x - \sum_{j = 1}^{N_x-1}W^1_j\cdot\varphi^1_j\Delta x\\
		& + O(\Delta t)\cdot\text{Area}(\Omega_1)\\
		& = -\iint\limits_{\Omega_1}\tilde{W}\cdot\partial_t\tilde{\varphi}dtdx - \int\limits_0^L\tilde{W}(t_1,x)\cdot\tilde{\varphi}(t_1,x)dx + O(\Delta t)\cdot\text{Area}(\Omega_1).
	\end{align*}
	Therefore taking the limit and applying the initial condition we have
	\begin{equation}
	\lim_{\Delta t,\Delta x\to 0}S_{t,1} = \iint\limits_{\Omega_1} w(t,x)\partial_t\varphi(t,x)dtdx - \int\limits_0^Lw(t_1,x)\varphi(t_1,x)dx.
	\end{equation}
	The nonlinear term $G(W^n_j)$ is models a distribution of order 0 for any $n$ and any $j$, so that the discrete summation is simple as
\begin{align*}
\lim_{\Delta t,\Delta x \to 0}S_{nonl} =\lim_{\Delta t,\Delta x \to 0}\sum_{n = 1}^{N_t}\sum_{j = 1}^{N_x}\mathbb{H}^n_jG(W^n_j)\cdot\varphi^n_j\Delta x\Delta t&\\= \lim_{\Delta t,\Delta x \to 0}\iint\limits_{\Omega_1}G(\tilde W)\cdot\tilde \varphi \ dxdt = \iint\limits_{\Omega_1}G(w)\cdot\varphi\ dxdt.
\end{align*}
	Suming up the 3 terms $S_x$, $S_t$ and $S_{nonl}$ and when $\Delta x, \Delta t \rightarrow 0$, we have $I(\tilde{W})$ in \eqref{discweakform} converges to 
	\begin{align*}
\iint\limits_{\Omega_0}w\cdot\partial_t\varphi + F(w)\cdot\partial_x&\varphi\ - G(w)\cdot \varphi \ dxdt +  \\ \nonumber
		& 2\int\limits_0^Lw(t_1,x)\cdot\varphi(t_1,x)dx + \int\limits_{t_1}^TF(w(t,L))\varphi(t,L)dt = 0
	\end{align*}
	so that Theorem \ref{main2} is proved.
\end{proof}

\section{Simulations}\label{simul}
In this section, we will realize some simulations of the Lax-Wendroff Scheme with the boundary conditions that are elaborated in Section \ref{bcS}. Specifically, we will present some results by varying the Courant number (see \cite{Roostaei2017} for details) and the mesh size of our model and finally achieve the convergence that we proved in Section \ref{weakS}.
\subsection*{Settings}
In order to show some features of the numerical method, we performed some simulations. We consider a pipe of length $L = 20$ [m] and diameter $D = 0.2$ [m]. We consider the constant $K = 1/\beta$, where we have set $\beta = 4 \times 10^{-9}$, so  $c = \sqrt{\frac{K}{\rho_a}} = 500$ [m/s] is the sound speed. The friction coefficient has been chosen to be $c_f = 2$ to show the velocity's dissipation within a few periods of the shockwave. We consider the atmospheric pressure $p_a = 1.01 \times 10^5$ [Pa]. The experiment is performed in the time interval $[0, 0.8]$ [s] and the time of closing is $t_1 = 0.04$ [s].  In the waterhammer equation we consider $C = \frac{c_f}{2D}$. 
The initial conditions are given by $\rho(0,x) = 1000$ [kg/m] and $v(0,x) = 1$ [m/s]. 
The boundary conditions we consider are:
\begin{itemize}
	\item Density:
	\begin{itemize}
		\item $\rho(t,0) = 1000$ \, \text{[kg/m$^3$]} \quad $t\geq 0$,
		\item $\rho(t, 2L) = 1000$ \, \text{[kg/m$^3$]} \quad $0\leq t < t_1$, \qquad $\partial_x \rho(t,L) = 0$ \, \text{[kg/m$^4$]} \quad $t \geq t_1$.
	\end{itemize}
	\item Velocity:
		\begin{itemize}
			\item $v(t,L) = 0$ \quad $t \ge t_1$,
			\item Other boundary conditions are by using the related formulas in Section \ref{bcS}.
		\end{itemize}
\end{itemize}

\subsection*{General behavior.}
In this section the mesh is set with $\Delta x = 0.1$ [m] and $\Delta t = \Courant \frac{\Delta x}{c}$ for $\alpha = 2$, so the Courant number is $\Courant = c \frac{\Delta t}{\Delta x} = 0.5$. We observe the velocity and the pressure. We can observe how the velocity decreases before the closing of the valve, due to the friction coefficient. The behavior changes suddenly at $t=t_1$ when the valve closes, and we see how this perturbation is propagated in both velocity and pressure in Figure \ref{heatmap}, while the velocity is still dissipating due to friction. We observe how the behavior tends to equilibrium as $t$ gets larger. 

%trim={<left> <lower> <right> <upper>}
\begin{figure}
	\centering
	\includegraphics[width = 0.6\textwidth, trim=0.5cm 4.4cm 0.5cm 4.2cm, clip]{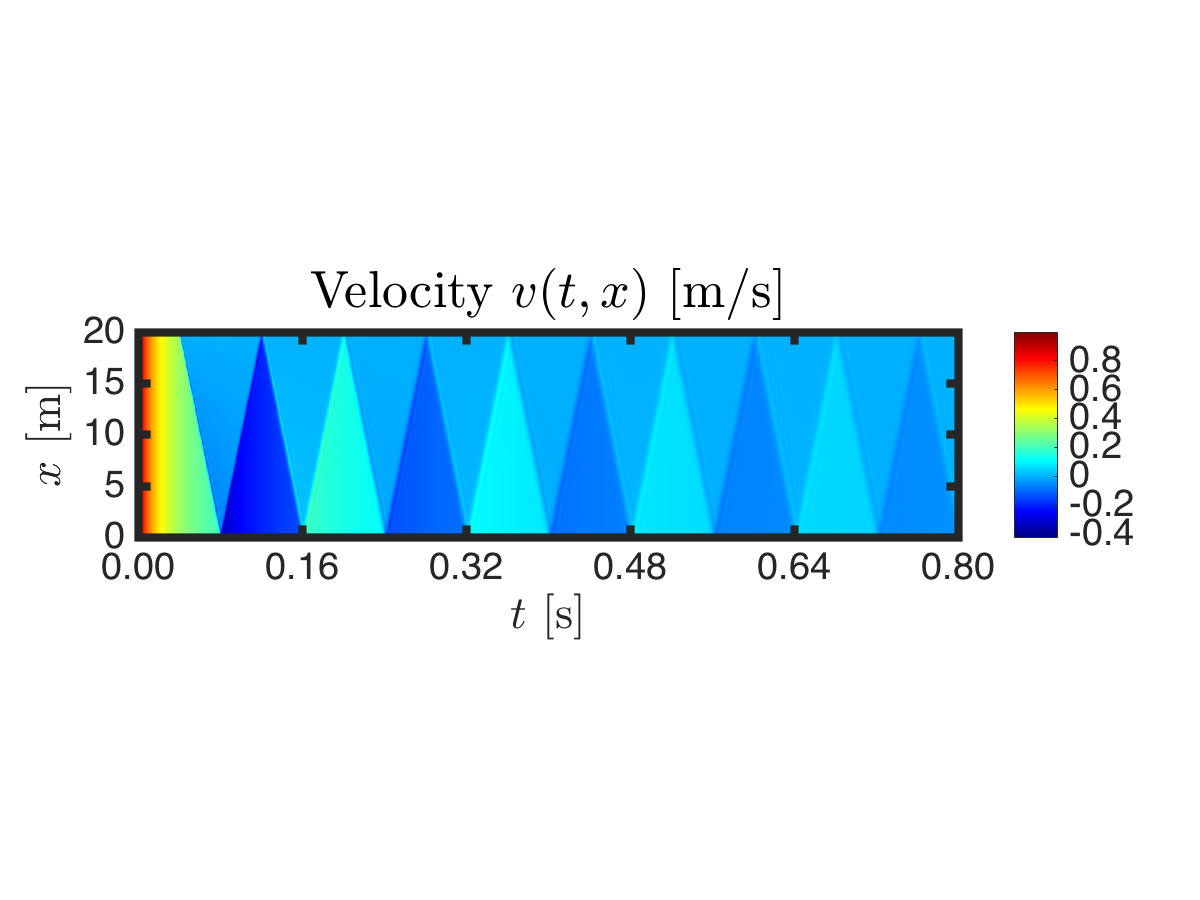}\\
	\includegraphics[width = 0.6\textwidth, trim=0.5cm 4.4cm 0.5cm 4.2cm, clip]{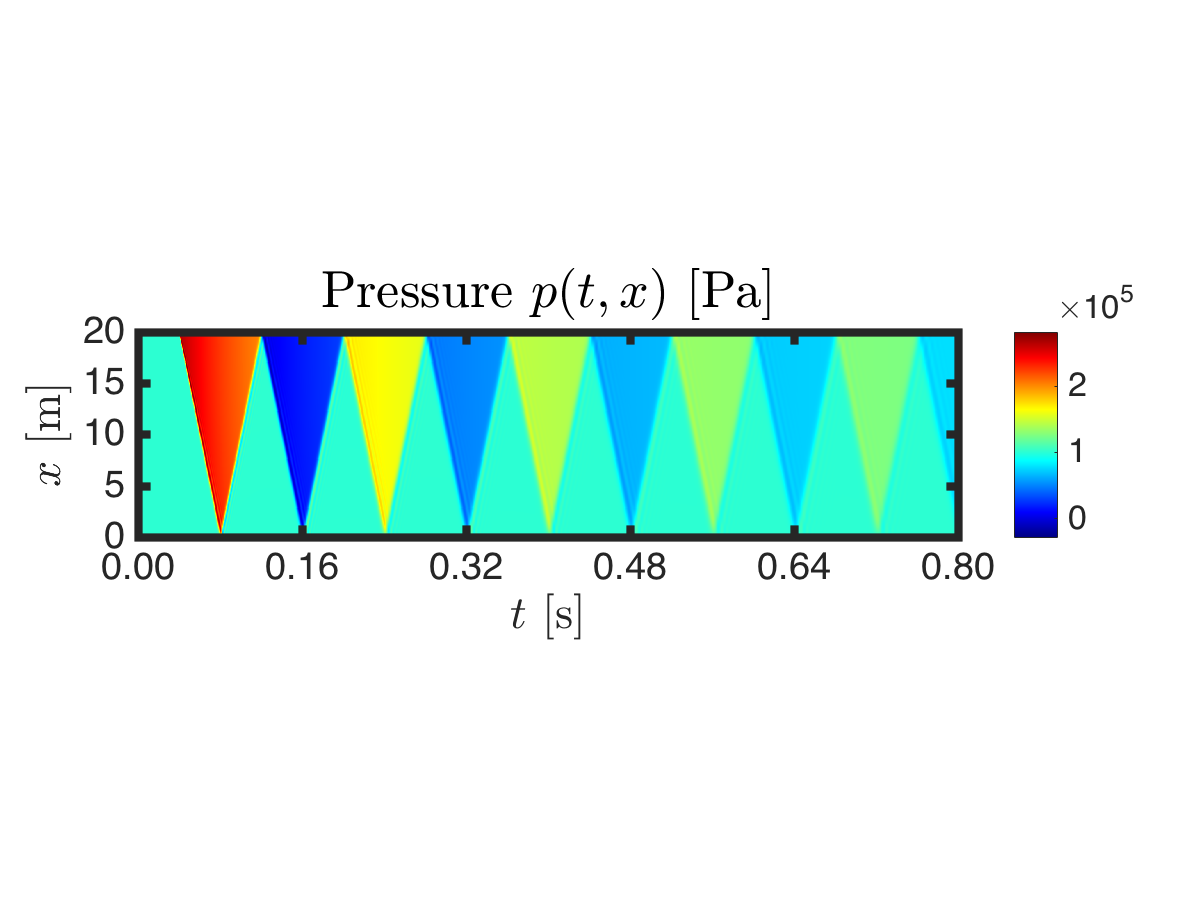}
	\caption{Heatmap of velocity and pressure after the valve closure}
	\label{heatmap}
\end{figure}
Next, we show the behavior of the velocity and pressure in the pipe, observed only at $x = \frac{L}{2}$. We see some oscillations (Figure \ref{behaviorvp}), which are due to the Lax-Wendroff scheme and we observe the numerical viscosity, see, e.g., \cite{Ramshaw1994numerical}. %The numerical viscosity is inevitable in the Lax-Wendroff scheme, which can be referred to . 

\begin{figure}
\centering
\includegraphics[width = 0.6\textwidth, trim=0.5cm 4.2cm 0.5cm 4.2cm, clip]{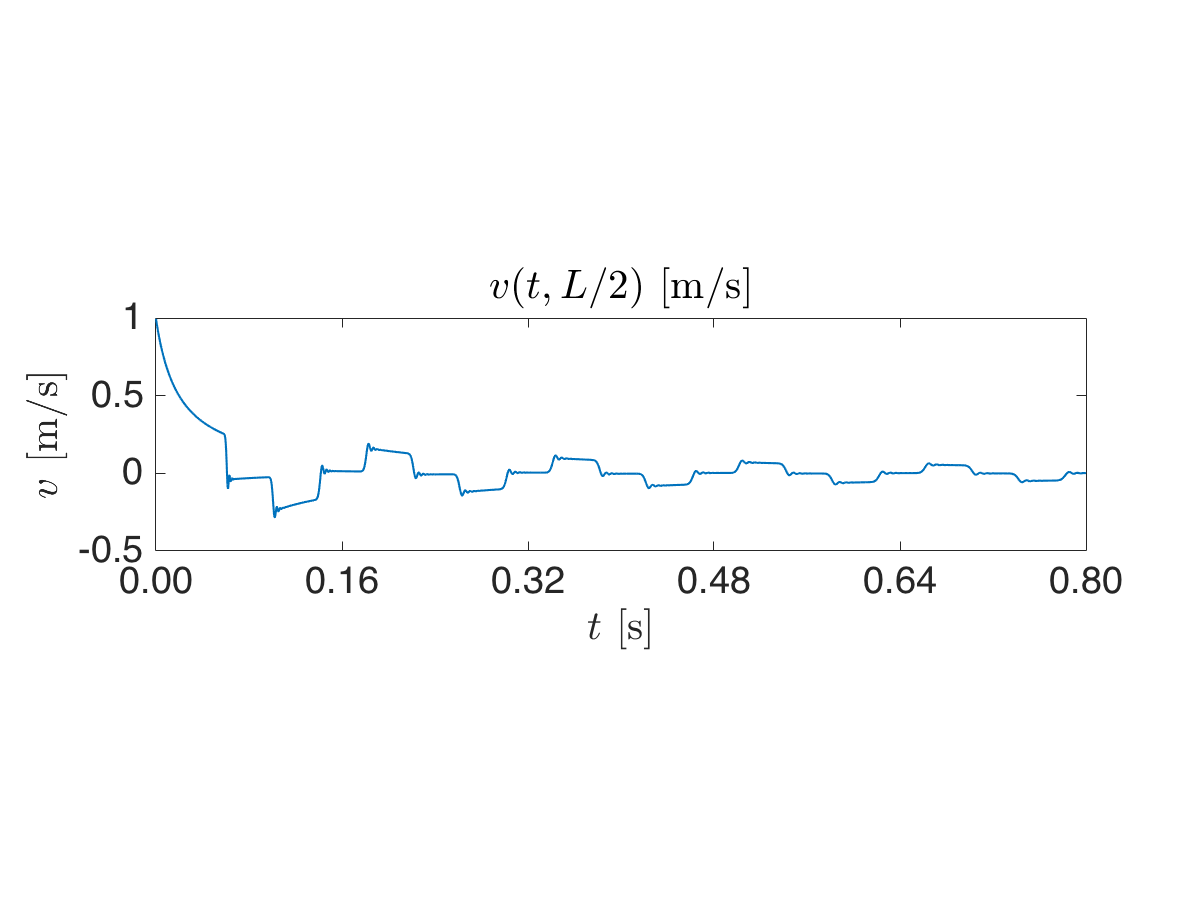}
\\
\includegraphics[width = 0.6\textwidth, trim=0.5cm 4.2cm 0.5cm 4.2cm, clip]{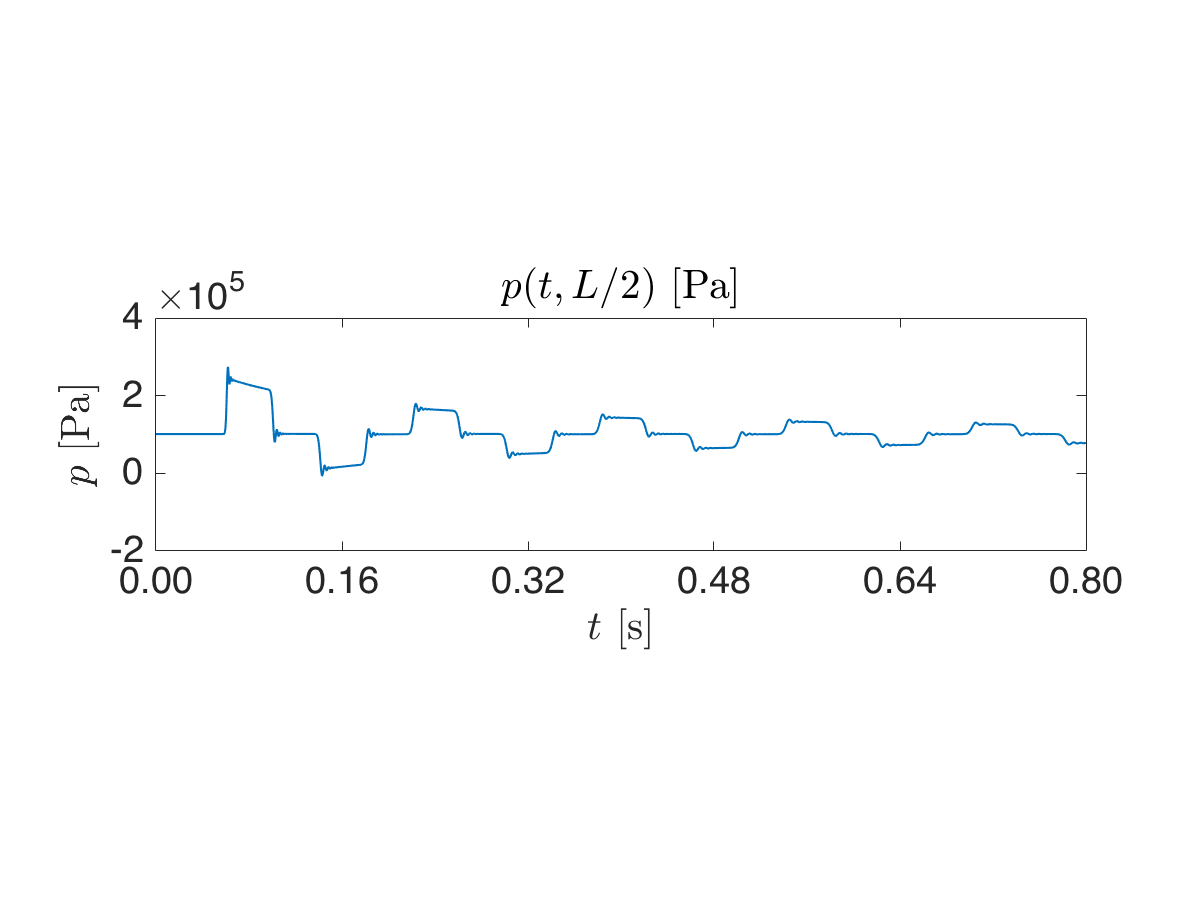}
\caption{Behavior of the velocity and pressure in the pipe.}
\label{behaviorvp}
\end{figure}

\subsection*{Vary the time step $\Delta t$}
The mesh is set with $\Delta x = 0.1$ [m] and $\Delta t = \Courant \frac{\Delta x}{c}$ for some values of $\Courant \leq 1$. 

Notice that the choice has been made such that $\frac{\Delta x}{\Delta t} \geq c$, which is a reasonable condition to make the scheme to ``see'' the physical behavior by ``being faster'' than it.

We observe in figure \ref{FigCourants} that for $\Courant = 1$ the scheme does not present oscillations but it does for another values, for example, $\Courant = 0.5$ or $\Courant = 0.1$. 
\begin{figure} 
	\centering
	\noindent
	\includegraphics[width = 0.45\textwidth, trim=0.5cm 1.8cm 0.5cm 2.5cm, clip]{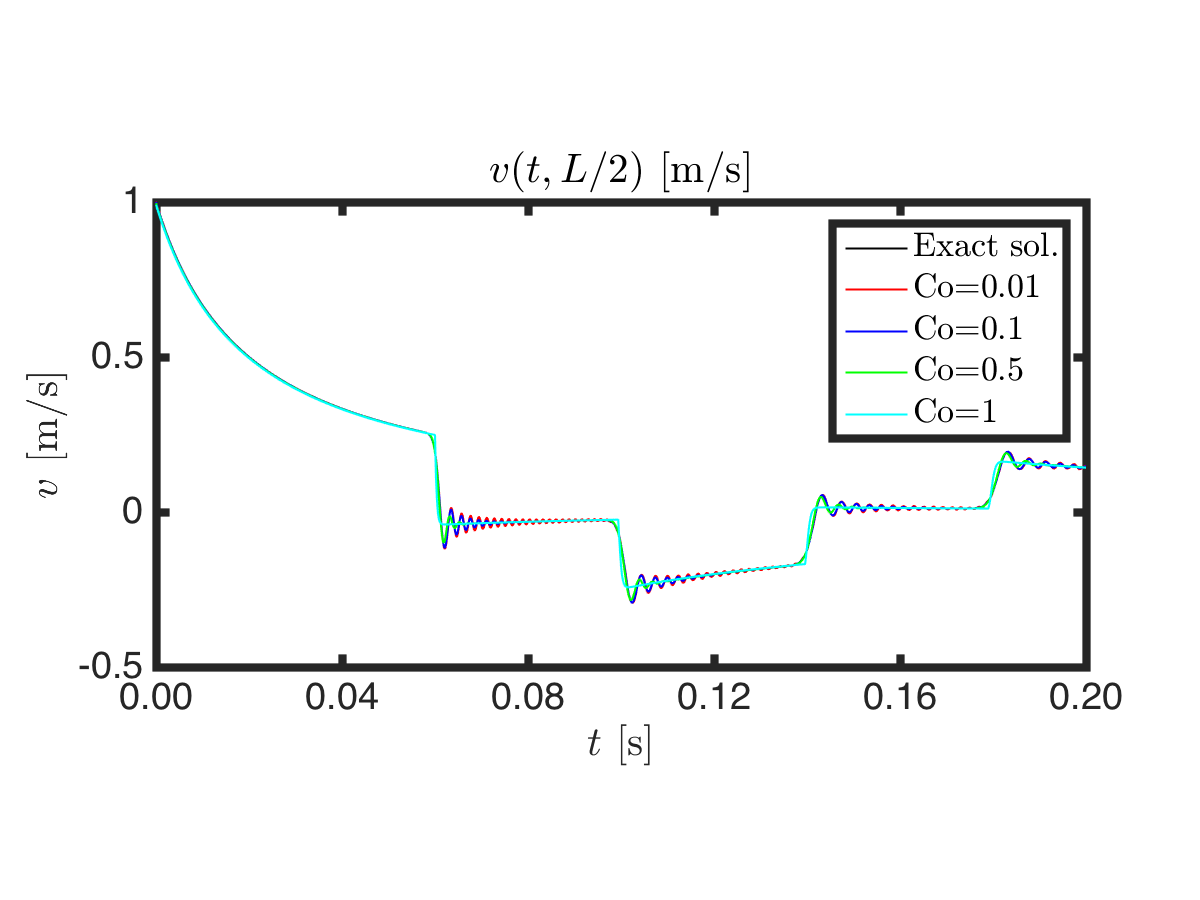}
	\includegraphics[width = 0.45\textwidth, trim=0.5cm 1.8cm 0.5cm 2.5cm, clip]{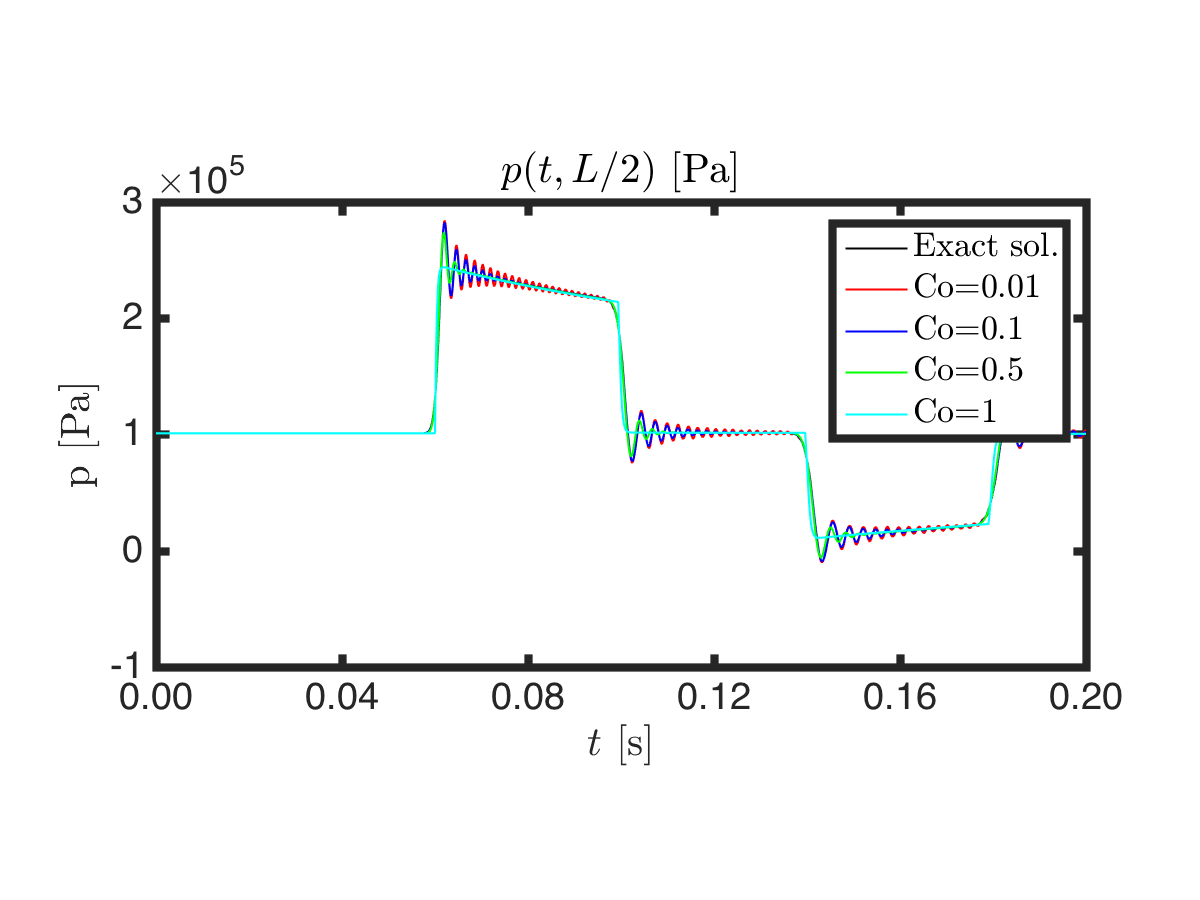}
	\caption{Behavior of $v$ and $p$ at $x=L/2$ by varying $\Courant$.}
	\label{FigCourants}
\end{figure}
\begin{remark}
Due to high computational cost, it is not recommendable to use much lower values of $\Courant$ if the time frame of interest is not much smaller than the one used in these experiments.
\end{remark}
\subsection*{Mesh sensitivity}
Since we have a result on convergence, we expect to see that if we use a finer mesh we get more accurate results. However, the computational cost could be too much if we use a much finer mesh. We did the following simulations,  in Figure \ref{deltax05}, in order to show that in practice we can reach accurate results for certain mesh sizes and making it finer will not improve the results significantly while increasing the computational cost too much. The mesh is set with $\Delta x \in \{0.5, 0.1, 0.05\}$ [m] and $\Delta t =  \frac{\Delta x}{10 c}$.
\begin{figure}
	\centering
	\noindent
	\includegraphics[width = 0.45\textwidth, trim=0.5cm 1.8cm 0.5cm 2.5cm, clip]{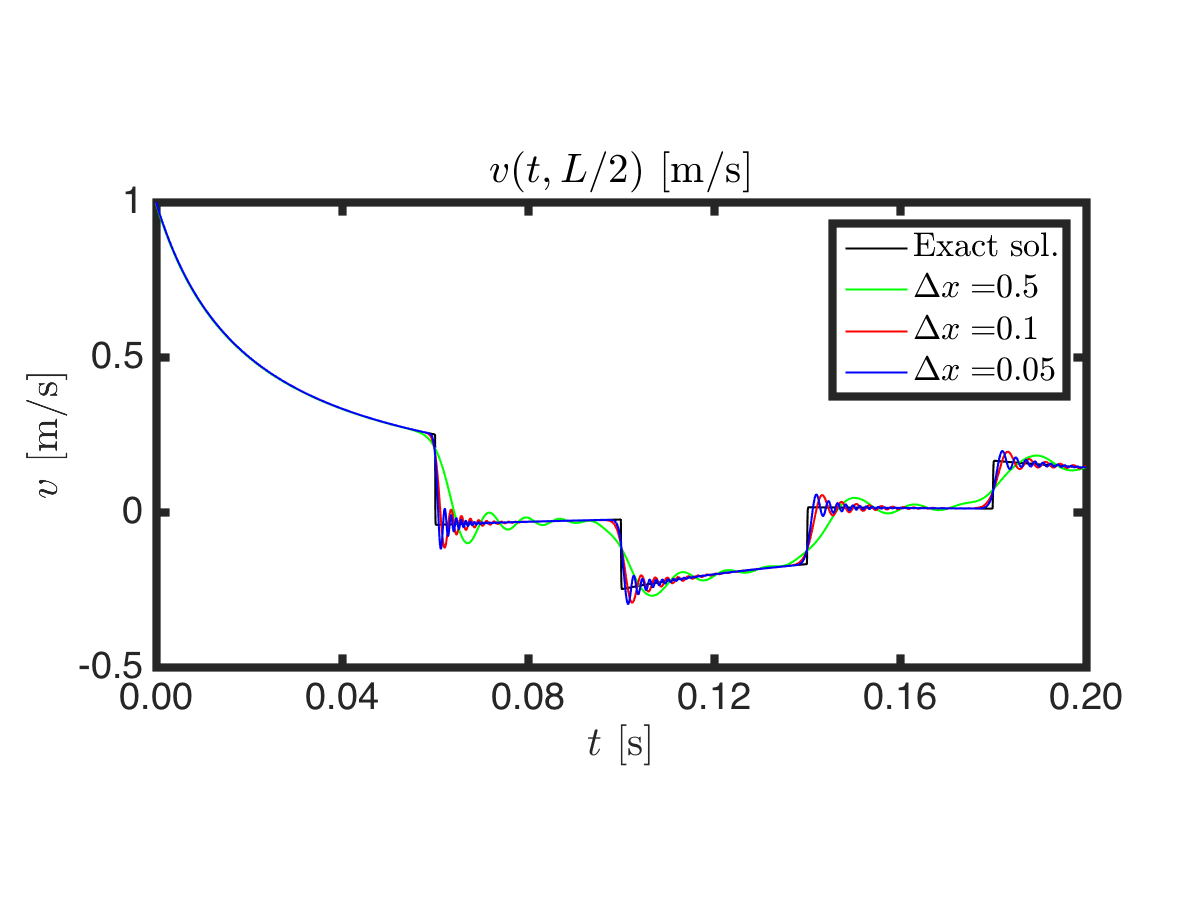}
	\includegraphics[width = 0.45\textwidth, trim=0.5cm 1.8cm 0.5cm 2.5cm, clip]{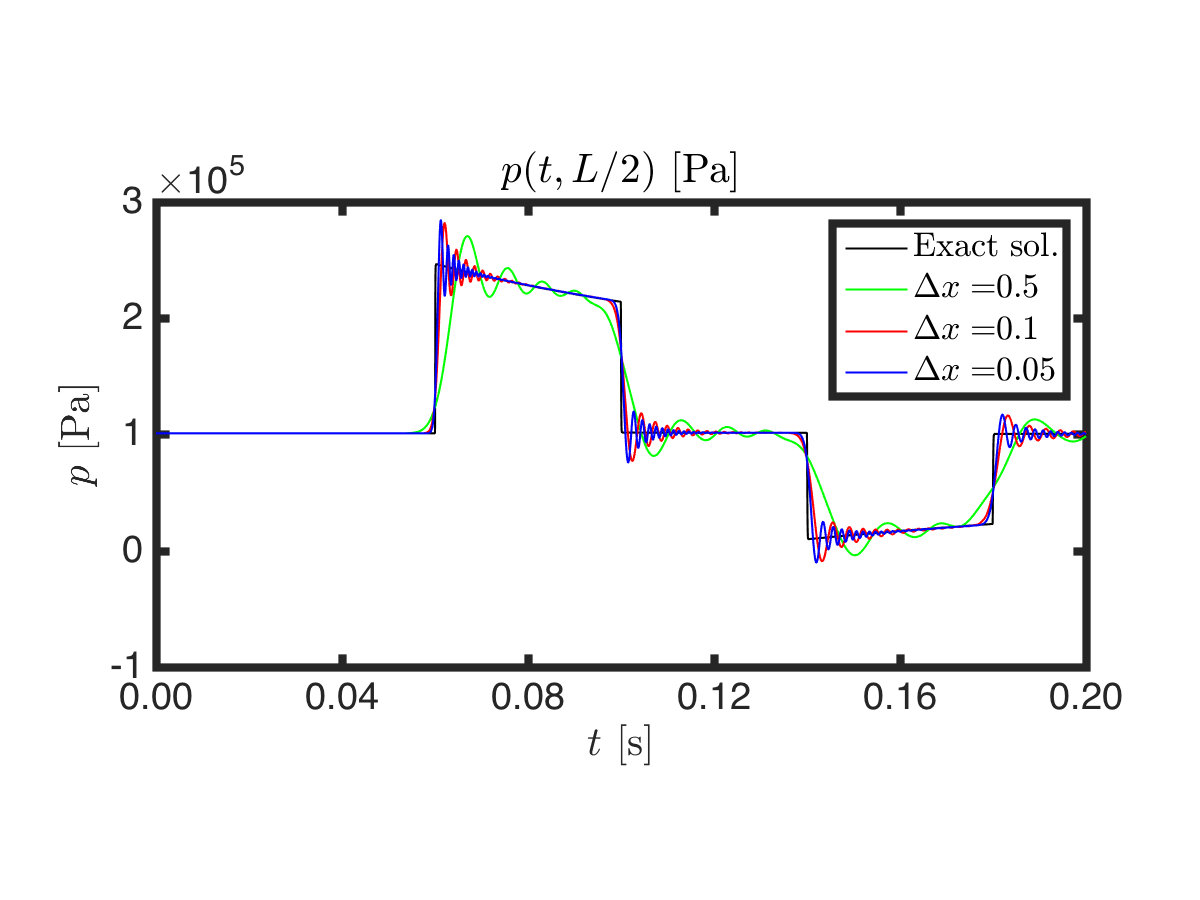}
	\caption{Behavior of $v$ and $p$ at $x = L/2$ by varying $\Delta x$.}
	\label{deltax05}
\end{figure}
\newline
\indent
We observe that while $\Delta x$ is smaller, the oscillations tend to lie on the shock line and the amplitude gets smaller, which is evidence of the convergence result.

\section{Conclusions and Future directions}
This article presents a novel approach to instantaneous valve closure in for the iso-thermal Euler equations or closed-channel Saint-Venant model. Lax-Wendroff schemes for steady states typically exhibit more oscillation than desirable and perhaps a marriage of two numerical techniques with the inverse Lax-Wendroff type boundary conditions could give improved results depending on the fluid flow which is modelled. Generic power semi-linearities for the full Euler equations have yet to be analyzed and should also be accessible via a generalization of Theorem \ref{main} to higher dimensions again using Strang's method from \cite{Strang1964}. Multiple pipes and shocks and higher moment models such as in  \cite{Koellermeier2020} also add an interesting challenge not yet covered by this work. 

\section{Acknowledgements}
The authors would like to thank Arthur Veldman, Stephan Trenn, and Axel Osses for interesting and useful discussions. H. Carrillo-Lincopi was funded by ANID / CORFO  International Centers of Excellence Program 10CEII-9157 Inria Chile, and Inria Challenge Oc\'eanIA.

\bibliographystyle{acm}
\bibliography{referencelib}

% \bib, bibdiv, biblist are defined by the amsrefs package.
\begin{bibdiv}
\begin{biblist}

\bib{Allairebook2007}{book}{
      author={Allaire, Gregoire},
       title={Numerical analysis and optimization: An introduction to
  mathematical modelling and numerical simulation},
   publisher={Oxford University Press; Illustrated edition},
        date={2007},
}

\bib{AudussePerthatme2004}{article}{
      author={Audusse, Emmanuel},
      author={Bouchut, Fran{\c c}ois},
      author={Bristeau, Marie-Odile},
      author={Klein, Rupert},
      author={Perthame, Benoit},
       title={A fast and stable well-balanced scheme with hydrostatic
  reconstruction for shallow water flows},
        date={200401},
     journal={Siam Journal on Scientific Computing},
      volume={25},
}

\bib{BastinCoron2016book}{book}{
      author={Bastin, Georges},
      author={Coron, Jean-Michel},
       title={Stability and boundary stabilization of 1-d hyperbolic systems},
        date={2016},
      volume={88},
}

\bib{BastinCoron2009}{article}{
      author={Bastin, Georges},
      author={Coron, Jean-Michel},
      author={Novel, Brigitte},
       title={On lyapunov stability of linearised saint-venant equations for a
  sloping channel},
        date={200906},
     journal={NHM},
      volume={4},
       pages={177\ndash 187},
}

\bib{baudin2005relaxation}{article}{
      author={Baudin, Micha{\"e}l},
      author={Berthon, Christophe},
      author={Coquel, Fr{\'e}d{\'e}ric},
      author={Masson, Roland},
      author={Tran, Quang~Huy},
       title={A relaxation method for two-phase flow models with hydrodynamic
  closure law},
        date={2005},
     journal={Numerische mathematik},
      volume={99},
      number={3},
       pages={411\ndash 440},
}

\bib{BealsMetivier1986}{article}{
      author={Beals, Michael},
      author={Metivier, Guy},
       title={Progressing wave solutions to certain nonlinear mixed problems},
        date={1986},
     journal={Duke Mathematical Journal},
      volume={53},
      number={1},
       pages={125 \ndash  137},
         url={https://doi.org/10.1215/S0012-7094-86-05307-X},
}

\bib{ChrisShuang2014}{book}{
      author={Christodoulou, Demetrios},
      author={Miao, Shuang},
       title={Compressible flow and euler's equations},
   publisher={Surveys in Modern Mathematics Volume 9, International Press of
  Boston, Incorporated},
        date={2014},
}

\bib{dafermos2005hyperbolic}{book}{
      author={Dafermos, Constantine~M.},
       title={Hyperbolic conservation laws in continuum physics},
   publisher={Springer},
        date={2005},
      volume={3},
}

\bib{Gallouet2022}{article}{
      author={Gallou\"{e}t, Thierry},
      author={Herbin, Rapha{\`e}le},
      author={Latch{\'e}, Jean-Claude},
       title={Lax--wendroff consistency of finite volume schemes for systems of
  non linear conservation laws: extension to staggered schemes},
        date={202206},
     journal={SeMA Journal},
       pages={333 \ndash  354},
}

\bib{gomez2017analysis}{misc}{
      author={G{\'o}mez, Crist{\'o}bal~R.},
       title={Analysis of blood flow in one dimensional elastic artery using
  navier-stokes conservation laws},
        date={2017},
}

\bib{Gugat2018}{article}{
      author={Gugat, Martin},
      author={Schultz, R\"{u}diger},
       title={Boundary feedback stabilization of the isothermal euler equations
  with uncertain boundary data},
        date={2018},
     journal={SIAM Journal on Control and Optimization},
      volume={56},
      number={2},
       pages={1491\ndash 1507},
}

\bib{Harten1989}{article}{
      author={Harten, Ami},
       title={Eno schemes with subcell resolution},
        date={1989},
     journal={Journal of Computational Physics},
      volume={83},
      number={1},
       pages={148\ndash 184},
}

\bib{Harten1997}{article}{
      author={Harten, Ami},
      author={Engquist, Bj{\"o}rn},
      author={Osher, Stanley},
      author={Chakravarthy, R.~Sukumar},
       title={Uniformly high order accurate essentially non-oscillatory
  schemes, iii},
        date={1997},
     journal={Journal of Computational Physics},
      volume={131},
      number={1},
       pages={3\ndash 47},
}

\bib{HartenOsher1986}{article}{
      author={Harten, Ami},
      author={Osher, Stanley},
      author={Engquist, Bj{\"o}rn},
      author={Chakravarthy, R.~Sukumar},
       title={Some results on uniformly high-order accurate essentially
  nonoscillatory schemes},
        date={1986},
     journal={Applied Numerical Mathematics},
      volume={2},
      number={3},
       pages={347\ndash 377},
}

\bib{Herbin2021}{article}{
      author={Herbin, Rapha{\`e}le},
      author={Latch{\'e}, Jean-Claude},
      author={Nasseri, Youssouf},
      author={Therme, Nicolas},
       title={A consistent quasi-second order staggered scheme for the
  two-dimensional shallow water equations},
        date={2021},
     journal={ArXiv},
      volume={abs/2111.09726},
}

\bib{Herty2009}{article}{
      author={Herty, Michael},
      author={Mohring, Jan},
      author={Schleper, Veronika},
       title={A new model for gas flow in pipe networks},
        date={200901},
     journal={Mathematical Methods in the Applied Sciences},
      volume={33},
       pages={845\ndash 855},
}

\bib{KausTren18}{inproceedings}{
      author={Kausar, Rukhsana},
      author={Trenn, Stephan},
       title={Water hammer modeling for water networks via hyperbolic pdes and
  switched daes},
        date={2018},
      editor={Klingenberg, Christian},
      editor={Westdickenberg, Michael},
   publisher={Springer},
     address={Cham},
       pages={123\ndash 135},
}

\bib{Koellermeier2020}{article}{
      author={Koellermeier, Julian},
      author={Rominger, Marvin},
       title={Analysis and numerical simulation of hyperbolic shallow water
  moment equations},
        date={202001},
     journal={Communications in Computational Physics},
      volume={28},
       pages={1038\ndash 1084},
}

\bib{liangfuyou2009}{article}{
      author={Liang, Fuyou},
      author={Himeno, Ryutaro},
      author={Liu, Hao},
       title={Biomechanical characterization of ventricular-arterial coupling
  during aging: A multi-scale model study},
        date={200905},
     journal={Journal of biomechanics},
      volume={42},
       pages={692\ndash 704},
}

\bib{LiuOsher1994}{article}{
      author={Liu, Xu-Dong},
      author={Osher, Stanley},
      author={Chan, Tony},
       title={Weighted essentially non-oscillatory schemes},
        date={1994},
     journal={Journal of Computational Physics},
      volume={115},
      number={1},
       pages={200\ndash 212},
}

\bib{LiuTadmor2008}{article}{
      author={Liu, Yingjie},
      author={Shu, Chi-Wang},
      author={Tadmor, Eitan},
      author={Zhang, Mengping},
       title={L2 stability analysis of the central discontinuous galerkin
  method and a comparison between the central and regular discontinuous
  galerkin methods},
        date={2008},
     journal={ESAIM: M2AN},
      volume={42},
      number={4},
       pages={593\ndash 607},
         url={https://doi.org/10.1051/m2an:2008018},
}

\bib{Luskin1982THEEO}{article}{
      author={Luskin, Mitchell},
      author={Temple, Blake},
       title={The existence of a global weak solution to the non-linear
  waterhammer problem},
        date={1982},
     journal={Communications on Pure and Applied Mathematics},
      volume={35},
       pages={697\ndash 735},
}

\bib{metivierlow}{article}{
      author={Metivier, Guy},
       title={The cauchy problem for semilinear hyperbolic systems with
  discontinuous data},
        date={1986},
     journal={Duke Mathematical Journal},
      volume={53},
      number={4},
       pages={983\ndash 1011},
}

\bib{quateroni2006}{article}{
      author={Quarteroni, Alfio},
       title={Cardiovascular mathematics},
        date={200901},
     journal={Proceedings oh the International Congress of Mathematicians, Vol.
  1, 2006-01-01, ISBN 978-3-03719-022-7, pags. 479-512},
      volume={1},
}

\bib{Ramshaw1994numerical}{article}{
      author={Ramshaw, John~D.},
       title={Numerical viscosities of difference schemes},
        date={1994},
     journal={Communications in Numerical Methods in Engineering},
      volume={10},
      number={11},
       pages={927\ndash 931},
      eprint={https://onlinelibrary.wiley.com/doi/pdf/10.1002/cnm.1640101108},
         url={https://onlinelibrary.wiley.com/doi/abs/10.1002/cnm.1640101108},
}

\bib{Roostaei2017}{article}{
      author={Roostaei, Morteza},
      author={Nouri, Alireza},
      author={Fattahpour, Vahidoddin},
      author={Chan, Dave},
       title={Evaluation of numercial schemes for capturing shock waves in
  modeling proppant transport in fractures},
        date={201711},
     journal={Petroleum Science},
      volume={14},
      number={4},
       pages={731\ndash 745},
}

\bib{Shu1998}{book}{
      author={Shu, Chi-Wang},
       title={Essentially non-oscillatory and weighted essentially
  non-oscillatory schemes for hyperbolic conservation laws},
        date={1998},
}

\bib{PerthameSimeoni2003}{inproceedings}{
      author={Simeoni, Chiara},
      author={Perthame, Benoit},
       title={Convergence of the upwind interface source method for hyperbolic
  conservation laws},
        date={200301},
}

\bib{SteinebachRosenSohr2012}{incollection}{
      author={Steinebach, Gerd},
      author={Rosen, Roland},
      author={Sohr, Annelie},
       title={Modeling and numerical simulation of pipe flow problems in water
  supply systems},
    language={en},
        date={2012},
      series={Martin, Klamroth et al. (Eds.): Mathematical Optimization of
  Water Networks. International Series of Numerical Mathematics, Vol. 162},
   publisher={Birkh{\"a}user},
     address={Basel},
       pages={3 \ndash  15},
}

\bib{stiriba2002nonlinear}{article}{
      author={Stiriba, Youssef},
       title={A nonlinear flux split method for hyperbolic conservation laws},
        date={2002},
     journal={Journal of Computational Physics},
      volume={176},
      number={1},
       pages={20\ndash 39},
}

\bib{Strang1964}{article}{
      author={Strang, Gilbert},
       title={Accurate partial difference methods},
        date={196412},
     journal={Numerische Mathematik},
       pages={37 \ndash  46},
}

\bib{Tadmor2012review}{article}{
      author={Tadmor, Eitan},
       title={A review of numerical methods for nonlinear partial differential
  equations.},
        date={2012},
        ISSN={02730979},
     journal={Bulletin (New Series) of the American Mathematical Society},
      volume={49},
      number={4},
       pages={507 \ndash  554},
  url={http://proxy-ub.rug.nl/login?url=https://search.ebscohost.com/login.aspx?direct=true&db=asx&AN=79349896&site=eds-live&scope=site},
}

\bib{Williams2019}{article}{
      author={Williams, David~M.},
       title={An analysis of discontinuous galerkin methods for the
  compressible euler equations: entropy and $l_2$ stability},
        date={201904},
     journal={Numerische Mathematik},
       pages={1079\ndash 1120},
}

\bib{WuZhao2015}{article}{
      author={Wu, Xiaoshuai},
      author={Zhao, Yuxin},
       title={A high-resolution hybrid scheme for hyperbolic conservation
  laws},
        date={2015},
     journal={International Journal for Numerical Methods in Fluids},
      volume={78},
      number={3},
       pages={162\ndash 187},
}

\bib{zhangboran2018}{article}{
      author={Zhang, Boran},
      author={Wan, Wuyi},
      author={Shi, Mengshan},
       title={Experimental and numerical simulation of water hammer in
  gravitational pipe flow with continuous air entrainment},
        date={2018},
        ISSN={2073-4441},
     journal={Water},
      volume={10},
      number={7},
}

\end{biblist}
\end{bibdiv}
\end{document}